\documentclass{amsart}

\usepackage{amsthm}
\usepackage{amssymb}
\usepackage{amsmath}
\usepackage{amsfonts}
\usepackage{amsrefs}
\usepackage{textcomp}
\usepackage[T1]{fontenc}
\usepackage[utf8x]{inputenc}
\usepackage{textcomp}
\usepackage{wasysym}
\usepackage{stmaryrd}
\usepackage{esint}
\usepackage[all]{xy}
\usepackage{graphicx}
\usepackage{bbm}
\usepackage{color}
\usepackage{hyperref}

\newtheorem{theorem}{Theorem}[section]

\newtheorem{remark}[theorem]{Remark}
\newtheorem{example}[theorem]{Example}
\newtheorem{lemma}[theorem]{Lemma}
\newtheorem{proposition}[theorem]{Proposition}
\newtheorem{corollary}[theorem]{Corollary}

\numberwithin{equation}{section}

\def\XXint#1#2#3{{\setbox0=\hbox{$#1{#2#3}{\int}$}
     \vcenter{\hbox{$#2#3$}}\kern-.5\wd0}}

\newcommand{\twopartdef}[4]
{
\left\{
		\begin{array}{ll}
			#1 & #2 \\
			#3 & #4
		\end{array}
	\right.
}

\newcommand{\threepartdef}[6]
{
	\left\{
		\begin{array}{lll}
			#1 & #2 \\
			#3 & #4 \\
			#5 & #6
		\end{array}
	\right.
}

\newcommand{\fourpartdef}[8]
{
	\left\{
		\begin{array}{lll}
			#1 & #2 \\
			#3 & #4 \\
			#5 & #6 \\
			#7 & #8
		\end{array}
	\right.
}

\usepackage{geometry}
\author{Wojciech G\'{o}rny}

\address{W. G\'orny: Faculty of Mathematics, University of Vienna, Oskar-Morgernstern-Platz 1, 1090 Wien, Austria; Faculty of Mathematics, Informatics and Mechanics, University of Warsaw, Banacha 2, 02-097 Warsaw, Poland.}

\email{wojciech.gorny@univie.ac.at}

\date{\today}

\subjclass[2020]{35J20, 35J25, 35J75, 49Q22}

\title{Applications of optimal transport methods in the least gradient problem}

\keywords{Least Gradient Problem, Optimal transport, SBV functions}

\begin{document}

\begin{abstract}
We study the consequences of the equivalence between the least gradient problem and a boundary-to-boundary optimal transport problem in two dimensions. We extend the relationship between the two problems to their respective dual problems, as well as prove several regularity and stability results for the least gradient problem using optimal transport techniques.
\end{abstract}

\maketitle

\section{Introduction}

The main goal of this paper is to study the relationship between the least gradient problem (see for instance \cite{BGG,Gor2018CVPDE,JMN,MRL,Mor,SWZ})
\begin{equation}\label{eq:leastgradientproblem}\tag{LGP}
\inf \bigg\{ \int_\Omega |Du|: \quad u \in BV(\Omega), \quad u|_{\partial\Omega} = g \bigg\},
\end{equation}
and the boundary-to-boundary Monge-Kantorovich optimal transport problem (see \cite{DS})
\begin{equation}\label{eq:kantorovich}\tag{KP}
\min \bigg\{ \int_{\overline{\Omega} \times \overline{\Omega}} |x-y| \, \mathrm{d}\gamma \, : \, \gamma \in  \mathcal{M}^+(\overline{\Omega} \times \overline{\Omega}), \, (\Pi_x)_{\#}\gamma = f^+ \,\, \mathrm{and} \,\, (\Pi_y)_{\#} \gamma = f^- \bigg\}.
\end{equation}
In the first problem, $g \in BV(\partial\Omega)$ and the boundary condition is understood as a trace of a BV function, while in the second problem $f^\pm \in \mathcal{M}^+(\partial\Omega)$ and a mass balance condition $f^+(\partial\Omega) = f^-(\partial\Omega)$ is satisfied. Let us stress that because the source and target measures are located on $\partial\Omega$, this is not the standard setting for the optimal transport problem, when at least one of the source and target measures is assumed to be absolutely continuous with respect to the Lebesgue measure (see for instance \cite{San2015,Vil}). We will study the two-dimensional situation: throughout the rest of the paper, unless noted otherwise, $\Omega \subset \mathbb{R}^2$ denotes an open bounded convex set. Then, the two problems are equivalent in the following sense: if $f^\pm = (\partial_\tau g)_\pm$, i.e. $f^+$ is a positive part of the tangential derivative of $g$ and $f^-$ is its negative part, then the infimal values coincide and from a solution of one problem we may construct a solution of the other problem (see \cite{GRS2017NA,DS}). The goal of this paper is to explore this relationship in more depth and use it to prove new regularity and stability results for solutions of the least gradient problem.

Let us briefly recall already known results in this direction. The least gradient problem in the form \eqref{eq:leastgradientproblem} has been first studied in \cite{SWZ} for continuous boundary data (although problems of this type have been studied earlier, see for instance \cite{BGG,KS,PP}). It appears naturally in the study of minimal surfaces, namely boundaries of superlevel sets of solutions to \eqref{eq:leastgradientproblem} area-minimising (see \cite{BGG}). Hence, it is not surprising that existence of solutions requires some geometric assumptions on $\Omega$. Indeed, existence of solutions for continuous boundary data has been proved in \cite{SWZ} under an assumption slightly weaker than strict convexity of $\Omega$. Since then, the theory developed in several directions: existence of solutions in the trace sense for less regular boundary data (see \cite{Gor2018CVPDE,Gor2019IUMJ,Mor}); anisotropic cases (see \cite{Gor2020NA,JMN,Zun}); non-strictly convex domains (see \cite{DG2019,RS,RS2}); the relaxed formulation for general $f \in L^1(\partial\Omega)$ and arbitrary Lipschitz domains (see \cite{MRL,Maz,Mor}); and extensions to metric measure spaces (see \cite{HKLS,KLLS}). In particular, in two dimensions we have existence of solutions on strictly convex domains for BV boundary data (see \cite{Gor2018CVPDE}), so both problems \eqref{eq:leastgradientproblem} and \eqref{eq:kantorovich} admit solutions for $g \in BV(\partial\Omega)$ and $f^{\pm} = (\partial_\tau g)_\pm \in \mathcal{M}^+(\partial\Omega)$.

The equivalence between the least gradient problem \eqref{eq:leastgradientproblem} and the boundary-to-boundary optimal transport problem \eqref{eq:kantorovich} on strictly convex domains in two dimensions was proved in two steps. First, it was proved in \cite{GRS2017NA} that on convex domains problem \eqref{eq:leastgradientproblem} is related to the Beckmann problem
\begin{equation}\label{eq:beckmannproblem}\tag{BP}
\inf \bigg\{ \int_\Omega |p|: p \in \mathcal{M}(\overline{\Omega}; \mathbb{R}^2), \quad \text{ div } p = f \bigg\},
\end{equation}
where the equation $\mathrm{div} \, p = f$ is understood in the distributional sense: for every $\phi \in C^1(\overline{\Omega})$, we have $\int_{\overline{\Omega}} \nabla \phi \cdot \mathrm{d}p = \int_{\partial\Omega}\phi\,\mathrm{d}f$. On strictly convex domains, problems \eqref{eq:leastgradientproblem} and \eqref{eq:beckmannproblem} are equivalent in the following sense: if $f \in \mathcal{M}(\partial\Omega)$ and $f = \partial_\tau g$, the infimal values coincide and given a minimiser of one problem, we may construct a minimiser of the other problem. The relationship between a minimiser $p \in \mathcal{M}(\overline{\Omega}; \mathbb{R}^2)$ to \eqref{eq:beckmannproblem} and a minimiser $u \in BV(\Omega)$ of \eqref{eq:leastgradientproblem} is given by $p = R_{-\frac{\pi}{2}} Du$. On convex domains, the situation is a bit more complicated, but some version of this equivalence still holds; for details we refer to Section \ref{sec:preliminaries}. On the other hand, on convex domains the Beckmann problem is known to be completely equivalent to the optimal transport problem \eqref{eq:kantorovich} (see \cite{San2015} and the references therein). This equivalence holds for general $f \in \mathcal{M}(\overline{\Omega})$; in fact, a typical assumption in the study of the Monge-Kantorovich problem \eqref{eq:kantorovich} is that either $f^+$ or $f^-$ is absolutely continuous with respect to the Lebesgue measure (see \cite{San2015,Vil}). Here, since both measures are supported on $\partial\Omega$, which is a set of zero Lebesgue measure, this is clearly not the case. The boundary-to-boundary problem \eqref{eq:kantorovich} was studied in depth in \cite{DS}; there, the authors proved that the equivalence with the least gradient problem holds also in anisotropic cases and that $L^p$ estimates for the transport density imply $W^{1,p}$ estimates in the corresponding least gradient problem. Since then, the optimal transport methods proved to be powerful tools in the study of the least gradient problem; the authors of \cite{DS} proved $L^p$ estimates for the transport density in a few settings on uniformly convex domains for $L^p$ boundary data, which imply $W^{1,p}$ regularity of solutions to \eqref{eq:leastgradientproblem} for $W^{1,p}$ boundary data with $p \leq 2$. Optimal transport methods were also used in \cite{DG2019} to study the least gradient problem on an annulus, which could not be handled by the previously known techniques due to the fact that its boundary is not connected.

The main goal of this paper is to study in more depth the equivalence between the two problems \eqref{eq:leastgradientproblem} and \eqref{eq:kantorovich}. We start by studying the dual problems to the least gradient problem and to the Monge-Kantorovich problem. The dual to problem \eqref{eq:kantorovich} is the well-known maximisation problem (see \cite{San2015,Vil})
\begin{equation}\label{eq:dualtokantorovich}\tag{dKP}
\sup\bigg\{\int_{\overline{\Omega}} \phi\,\mathrm{d}(f^+ - f^-)\,:\,\phi \in  \mathrm{Lip}_1(\overline{\Omega})\bigg\}. 
\end{equation}
whose solutions are known as Kantorovich potentials, while the dual problem to \eqref{eq:leastgradientproblem} is the following maximisation problem (see \cite{Gor2020TAMS,Mor})
\begin{equation}\label{eq:dualtolgp}\tag{dLGP}
\sup \bigg\{ \int_{\partial\Omega} [\mathbf{z},\nu] \, g \, \mathrm{d}\mathcal{H}^{1}: \mathbf{z} \in \mathcal{Z} \bigg\},
\end{equation}
where
\begin{equation*}
\mathcal{Z} = \bigg\{ \mathbf{z} \in L^\infty(\Omega;\mathbb{R}^2), \quad \mathrm{div}(\mathbf{z}) = 0, \quad \| \mathbf{z} \|_\infty \leq 1 \mbox{ a.e. in } \Omega \bigg\}
\end{equation*}
and $[\mathbf{z},\nu]$ denotes the normal trace of a vector field whose divergence is a Radon measure (see \cite{Anz,CF}). It turns out that these problems are again equivalent, in the sense that their supremal values coincide and from a solution to one of the two problems we may construct a solution of the other problem; this is proved in Theorem \ref{thm:dualproblems}. Then, we exploit this relationship to study the structure of solutions to the least gradient problem.

Then, we use the equivalence between problems \eqref{eq:leastgradientproblem} and \eqref{eq:kantorovich} for two purposes. The first one is to give new results on the regularity of solutions to the least gradient problem. In contrast to the results in \cite{DS}, we focus on the case when the boundary datum is discontinuous, and the optimal transport plan in \eqref{eq:kantorovich} is not necessarily unique nor induced by a map. The main result in this direction is Theorem \ref{thm:sbvregularity}, which states that for $g \in SBV(\partial\Omega)$, then even though solutions to problem \eqref{eq:leastgradientproblem} may no longer be unique, then every solution lies in $SBV(\Omega)$. It is complemented by a few results on local properties of solutions in the case when the jump set is finite.

The second purpose is to study stability of families of solutions to the least gradient problem. We study two main cases: in the first one, we approximate the boundary datum in the strict topology of $BV(\partial\Omega)$, and prove that then the sequence of solutions converges in the strict topology of $BV(\Omega)$ to a solution of the original problem; this is done in Theorem \ref{thm:stability}. In the second one, we instead approximate the domain by a decreasing sequence of strictly convex domains converging in the Hausdorff distance, with particular applications to the case when $\Omega$ is only convex. In the process, we also prove an estimate on the total variation of the solution interesting in its own right.

The structure of the paper is as follows. In Section \ref{sec:preliminaries}, we recall known results about the equivalence between problems \eqref{eq:leastgradientproblem} and \eqref{eq:kantorovich} and some basic properties of solutions to both problems. In Section \ref{sec:dualproblems}, we study the relationship between the problems \eqref{eq:dualtolgp} and \eqref{eq:dualtokantorovich}, namely the duals of the original problems \eqref{eq:leastgradientproblem} and \eqref{eq:kantorovich}, and its consequences for the structure of solutions to the least gradient problem. Then, in Section \ref{sec:SBVregularity} we study the consequences of the equivalence between the two problems for the regularity of solutions to the least gradient problem. In the final Section \ref{sec:stability} we study stability properties of sequences of solutions to approximate problems. Finally, let us note that all the results are valid also in the anisotropic case, when the distance is constructed from a strictly convex norm. Nonetheless, in order to simplify the notation, throughout the paper we will give the results for the Euclidean norm and only comment on the anisotropic case at the end of each Section; we do this partly because the results are already new in the Euclidean case and their validity in the anisotropic case is simply a useful byproduct of the proofs.

\section{Preliminaries}\label{sec:preliminaries}

In this Section, we present the equivalences between the least gradient problem \eqref{eq:leastgradientproblem}, the Beckmann problem \eqref{eq:beckmannproblem} and the classical Monge-Kantorovich problem \eqref{eq:kantorovich}, together with some properties of solutions to these problems. Recall that throughout the paper, unless noted otherwise, $\Omega \subset \mathbb{R}^2$ denotes an open bounded convex set.

First, let us focus on the relationship between the least gradient problem \eqref{eq:leastgradientproblem} 
\begin{equation*}
\min\bigg\{ \int_\Omega |Du| \,:\,u \in BV(\Omega),\, u|_{\partial\Omega} = g \bigg\},
\end{equation*}
and the Beckmann problem \eqref{eq:beckmannproblem}
\begin{equation*}
\min \bigg\{ \int_{\overline{\Omega}} |p|\,:\, p \in \mathcal{M}(\overline{\Omega}; \mathbb{R}^2),\,\, \mathrm{div} \, p = f \bigg\}.
\end{equation*}
The boundary condition in \eqref{eq:leastgradientproblem} is understood as a trace of a $BV$ function and the divergence condition in \eqref{eq:beckmannproblem} is understood in the distributional sense. In other words, we have $\mathrm{div} \, p = 0$ in $\Omega$ and $[p, \nu] = f$ on $\partial\Omega$; here, $[p,\nu]$ denotes the normal trace of a vector field whose divergence is a Radon measure (for a precise definition see \cite{Anz} or \cite{CF}). Suppose that $g \in BV(\partial\Omega)$ and suppose that $f \in \mathcal{M}(\partial\Omega)$ satisfies a mass balance condition, i.e. $\int_{\partial\Omega} \mathrm{d}f = 0$. Then, both problems admit solutions (see \cite{Gor2018CVPDE} for the least gradient problem and \cite{San2015} for the Beckmann problem). Notice that such $f$ and $g$ are in a one-to-one correspondence (up to an additive constant in $g$) via the relation $f = \partial_\tau g$.

It was observed in \cite{GRS2017NA} that problems \eqref{eq:leastgradientproblem} and \eqref{eq:beckmannproblem} are closely related. Namely, if we take an admissible function $u \in BV(\Omega)$ in \eqref{eq:leastgradientproblem}, then $p = R_{-\frac{\pi}{2}} Du$ is admissible in \eqref{eq:beckmannproblem}. Indeed, in dimension two a rotation of a gradient by $-\frac{\pi}{2}$ is a divergence-free field in $\Omega$ and it interchanges the normal and tangent components at the boundary. In the other direction, given a vector field $p \in L^1(\Omega;\mathbb{R}^2)$ admissible in \eqref{eq:beckmannproblem}, we can recover $u \in W^{1,1}(\Omega)$ admissible in \eqref{eq:leastgradientproblem}; the following result has been proved in \cite[Proposition 2.1]{GRS2017NA}.

\begin{proposition}\label{prop:recoveryinL1}
Suppose that $p \in L^1(\Omega; \mathbb{R}^2)$ and $\mathrm{div} \, p = 0$ in the sense of distributions. Then, there exists $u \in W^{1,1}(\Omega)$ such that $p = R_{-\frac{\pi}{2}} \nabla u$. Moreover, we have $[p,\nu] = \partial_\tau(Tu)$.
\end{proposition}

In \cite{DS}, the authors noticed that this result may be improved: if $p \in \mathcal{M}(\overline{\Omega}; \mathbb{R}^2)$ is such that $|p|(\partial\Omega) = 0$, then there exists $u \in BV(\Omega)$ such that $p = R_{-\frac{\pi}{2}} Du$ and $[p,\nu] = \partial_\tau(Tu)$. In particular, notice that $|p| = |Du|$ as measures on $\overline{\Omega}$; hence, the infimal values in \eqref{eq:leastgradientproblem} and \eqref{eq:beckmannproblem} coincide. Let us sum up these considerations as follows.

\begin{theorem}\label{thm:lgpbeckmannequivalence} Let $\Omega\subset \mathbb{R}^2$ be an open bounded convex set. Then, the problems \eqref{eq:leastgradientproblem} and \eqref{eq:beckmannproblem} are equivalent in the following sense: \\
(1) Their infimal values coincide, i.e. $\inf \eqref{eq:leastgradientproblem} = \inf \eqref{eq:beckmannproblem}$; \\
(2) Given a solution $u \in BV(\Omega)$ of \eqref{eq:leastgradientproblem}, we can construct a solution $p \in \mathcal{M}(\overline{\Omega}; \mathbb{R}^2)$ of \eqref{eq:beckmannproblem}; moreover, $p = R_{-\frac{\pi}{2}} Du$; \\
(3) Given a solution $p \in \mathcal{M}(\overline{\Omega}; \mathbb{R}^2)$ of \eqref{eq:beckmannproblem} with $|p|(\partial\Omega) = 0$, we can construct a solution $u \in BV(\Omega)$ of \eqref{eq:leastgradientproblem}; moreover, $p = R_{-\frac{\pi}{2}} Du$.
\end{theorem}

Now, we turn to the equivalence between the Beckmann problem \eqref{eq:beckmannproblem}
\begin{equation*}
\min \bigg\{ \int_{\overline{\Omega}} |p|\,:\, p \in \mathcal{M}(\overline{\Omega}; \mathbb{R}^2),\,\, \mathrm{div} \, p = f \bigg\}.
\end{equation*}
and the Monge-Kantorovich problem \eqref{eq:kantorovich}
\begin{equation*}
\min \bigg\{ \int_{\overline{\Omega} \times \overline{\Omega}} |x-y| \, \mathrm{d}\gamma \, : \, \gamma \in \mathcal{M}^+(\overline{\Omega} \times \overline{\Omega}), \, (\Pi_x)_{\#}\gamma = f^+ \,\, \mathrm{and} \,\, (\Pi_y)_{\#} \gamma = f^- \bigg\}.
\end{equation*}
This equivalence is standard in the optimal transport theory (see for instance \cite[Chapter 4]{San2015}), but we will describe it shortly for completeness. Here, $\Omega \subset \mathbb{R}^d$ is a bounded open convex set, $f \in \mathcal{M}(\overline{\Omega})$ and $f = f^+ - f^-$ is its decomposition into a positive and negative part. In order for the problem to be well defined, we assume that the mass balance condition $f^+(\overline{\Omega}) = f^-(\overline{\Omega})$ holds. In particular, this setting covers our case, when the measures are concentrated on $\partial\Omega$.

First, let us recall a few standard results in the optimal transport theory. The fact that \eqref{eq:dualtokantorovich} is the dual problem to \eqref{eq:kantorovich} is well-known, see for instance \cite{San2015,Vil}. As a corollary of its proof, we get that there exist solutions to both problems, and that any optimal transport plan $\gamma$ and any Kantorovich potential $\phi$ satisfy the following equality:
$$\int_{\overline{\Omega} \times \overline{\Omega}} (|x-y| - (\phi(x) - \phi(y))) \, \mathrm{d}\gamma(x,y) =0,$$
which implies that 
$$\phi(x) - \phi(y) = |x-y| \quad \mbox{on supp}(\gamma).$$
If $\phi$ is a Kantorovich potential, we call any maximal segment \,$\left[x,y\right]$ satisfying \,$\phi(x)-\phi(y)=|x-y|$ a 
{\it transport ray}. The Kantorovich potential is in general not unique, but frame of transport rays does not depend on the choice of $\phi$ (at least on the support of $\gamma$). Moreover, an optimal transport plan $\gamma$ has to move the mass along the transport rays.

Now, we turn our attention to the equivalence between the Kantorovich problem \eqref{eq:kantorovich} and the Beckmann problem \eqref{eq:beckmannproblem}. First, let us see that for any $p \in \mathcal{M}(\overline{\Omega};\mathbb{R}^d)$ admissible in the Beckmann problem \eqref{eq:beckmannproblem}, we have that for any $C^1$ function $\phi$ with $|\nabla \phi| \leq 1$
\begin{equation*}
|p|(\overline{\Omega}) = \int_{\overline{\Omega}} 1 \, \mathrm{d}|p| \geq \int_{\overline{\Omega}} (-\nabla \phi) \cdot \mathrm{d}p = \int_{\overline{\Omega}} \phi \, \mathrm{d}f,
\end{equation*}
hence $\min \eqref{eq:beckmannproblem} \geq \max \eqref{eq:dualtokantorovich} = \min \eqref{eq:kantorovich}$ (the supremum in \eqref{eq:dualtokantorovich} is taken for Lipschitz functions, but we may approximate them uniformly by $C^1$ functions). On the other hand, given an optimal transport plan $\gamma$, we may construct a vector measure $p_\gamma \in \mathcal{M}(\overline{\Omega};\mathbb{R}^d)$ defined by the formula
\begin{equation}\label{eq:definitionofpgamma}
\langle p_\gamma,\varphi \rangle :=\int_{\overline{\Omega} \times \overline{\Omega}} \int_0^1 \omega_{x,y}'(t) \cdot \varphi(\omega_{x,y}(t)) \, \mathrm{d}t \, \mathrm{d}\gamma (x,y)
\end{equation}
for all $\varphi \in C(\overline{\Omega}; \mathbb{R}^d)$. Here, $\omega_{x,y}(t) = (1-t)x + ty$ is the constant-speed parametrisation of $[x,y]$. By taking $\varphi = \nabla \phi$, it is immediate that $p_\gamma$ satisfies the divergence constraint. The total mass of $p_\gamma$ will be estimated using the {\it transport density} $\sigma_\gamma \in \mathcal{M}(\overline{\Omega})$, which is defined by the formula
\begin{equation}\label{eq:definitionofsigma}
\langle \sigma_\gamma,\phi \rangle := \int_{\overline{\Omega} \times \overline{\Omega}} \int_0^1 |\omega_{x,y}'(t)| \, \phi (\omega_{x,y}(t)) \, \mathrm{d}t \, \mathrm{d}\gamma (x,y)
\end{equation}
for all $\phi \in C(\overline{\Omega})$. The vector measure $p_\gamma$ and the scalar measure $\sigma_\gamma$ are related in the following way: if $u$ is a Kantorovich potential, we have
\begin{equation*}
\omega_{x,y}'(t) = y - x = - |x-y| \frac{x-y}{|x-y|} = - |x-y| \nabla u(\omega_{x,y}(t))
\end{equation*}
for all $t \in (0,1)$ and $x,y \in \mathrm{supp}(\gamma)$. Thus, $\langle p_\gamma, \varphi \rangle = \langle \sigma_\gamma, -\varphi \cdot \nabla u \rangle$, so
\begin{equation}
p_\gamma = -\nabla u \cdot \sigma_\gamma.
\end{equation}
Hence, $p_\gamma$ is absolutely continuous with respect to $\sigma_\gamma$ and $|p_\gamma| \leq \sigma_\gamma$. Hence,
\begin{equation*}
\min \eqref{eq:kantorovich} = \int_{\overline{\Omega} \times \overline{\Omega}} |x-y| \, \mathrm{d}\gamma = \int_{\overline{\Omega} \times \overline{\Omega}} \int_0^1 |\omega_{x,y}'(t)| \, \mathrm{d}t \, \mathrm{d}\gamma(x,y) = \sigma_\gamma(\overline{\Omega}) \geq |w_\gamma|(\overline{\Omega}) \geq \min \eqref{eq:beckmannproblem}.
\end{equation*}
Hence, $\min \eqref{eq:kantorovich} = \min \eqref{eq:beckmannproblem}$, and from an optimal transport plan $\gamma$ we can construct a solution to the Beckmann problem \eqref{eq:beckmannproblem}. Moreover, we have $|p_\gamma| = \sigma_\gamma$. On the other hand, it can be shown that every solution to \eqref{eq:beckmannproblem} is of the form $p = p_\gamma$ for some optimal transport plan $\gamma$, see \cite[Theorem 4.13]{San2015}. We summarise the above discussion in the following Theorem.

\begin{theorem}\label{thm:beckmannkantorovichequivalence} Let $\Omega\subset \mathbb{R}^d$ be an open bounded convex set. Then, the problems \eqref{eq:kantorovich} and \eqref{eq:beckmannproblem} both admit solutions, and are equivalent in the following sense: \\
(1) Their minimal values coincide, i.e. $\min \eqref{eq:kantorovich} = \min \eqref{eq:beckmannproblem}$; \\
(2) Given an optimal transport plan $\gamma \in \mathcal{M}^+(\overline{\Omega} \times \overline{\Omega})$ in \eqref{eq:kantorovich}, we can construct a solution $p_\gamma \in \mathcal{M}(\overline{\Omega}; \mathbb{R}^d)$ to \eqref{eq:beckmannproblem}; moreover, $|p_\gamma| = \sigma_\gamma$; \\
(3) Given a solution $p \in \mathcal{M}(\overline{\Omega}; \mathbb{R}^2)$ to \eqref{eq:beckmannproblem}, we can construct an optimal transport plan $\gamma \in \mathcal{M}^+(\overline{\Omega} \times \overline{\Omega})$ in \eqref{eq:kantorovich} such that $p = p_\gamma$; moreover, $|p_\gamma| = \sigma_\gamma$.
\end{theorem}

The equivalence between the two-dimensional least gradient problem \eqref{eq:leastgradientproblem} and the Monge-Kantorovich problem \eqref{eq:kantorovich} comes from combining Theorems \ref{thm:lgpbeckmannequivalence} and \ref{thm:beckmannkantorovichequivalence}. Given a solution to \eqref{eq:leastgradientproblem}, we may construct an optimal transport plan $\gamma$ for \eqref{eq:kantorovich} with $f^\pm = (\partial_\tau g)_\pm$; in the other direction, since $|p_\gamma| = \sigma_\gamma$, we may recover a solution to \eqref{eq:leastgradientproblem} from an optimal transport plan $\gamma$ as soon as the transport density $\sigma_\gamma$ gives no mass to the boundary, i.e. $\sigma_\gamma(\partial\Omega) = 0$. An important special case is when $\Omega$ is strictly convex. Using an equivalent formula for transport density, i.e. for every Borel set $A \subset \overline{\Omega}$
\begin{equation}\label{eq:definitionofsigmaversiontwo}
\sigma_\gamma(A) = \int_{\overline{\Omega} \times \overline{\Omega}} \mathcal{H}^1([x,y] \cap A) \, \mathrm{d}\gamma(x,y),
\end{equation}
we have that if $\Omega$ is strictly convex, then for any optimal transport plan we have that $\sigma_\gamma(\partial\Omega) = 0$, and the correspondence between problems is one-to-one.

This link between the two problems was exploited for the first time in \cite{DS}. For a strictly convex domain $\Omega$, the authors studied the boundary-to-boundary optimal transport problem. The setting is a bit unusual, since a standard assumption used to prove estimates on the transport density is that either $f^+$ or $f^-$ is in $L^1(\Omega)$; here, this is clearly not the case, since both measures are supported on the boundary $\partial\Omega$. Nonetheless, the authors showed that if at least one of the measures $f^\pm$ is atomless, then the optimal transport plan is unique and induced by a map, and proved several variants of regularity estimates on the transport density; let us recall the version most useful to us in the course of the paper (see \cite[Proposition 3.2, Remark 3.4]{DS}; for the exact definition of uniform convexity, see Section \ref{sec:SBVregularity}).

\begin{theorem}\label{thm:dweiksantambrogio}
Suppose that $\Omega \subset \mathbb{R}^d$ is an open, bounded, uniformly convex set. Suppose that $f^+ \in L^p(\partial\Omega)$ with $p \in [1,2)$. Let $\gamma$ be the optimal transport plan between $f^+$ and some $f^- \in \mathcal{M}^+(\partial\Omega)$. Then, $\sigma_\gamma \in L^p(\Omega)$.
\end{theorem}

The main application of this result so far is the $W^{1,p}$ regularity of solutions to the least gradient problem. Suppose that $d = 2$ and $g \in W^{1,p}(\partial\Omega)$ for $p \in [1,2)$, where $g$ is the boundary datum in problem \eqref{eq:leastgradientproblem}; then, $f = \partial_\tau g \in L^p(\partial\Omega)$. By Theorem \ref{thm:dweiksantambrogio}, we have $\sigma_\gamma \in L^p(\Omega)$. Let $u \in BV(\Omega)$ be the unique solution to \eqref{eq:leastgradientproblem}, constructed from the unique optimal transport plan $\gamma$. Since we have $|Du| = |p_\gamma| = \sigma_\gamma$, we actually have $\nabla u \in L^p(\Omega)$; since by a maximum principle we also have $u \in L^\infty(\Omega)$, we have that $u \in W^{1,p}(\Omega)$.

Finally, let us note that the equivalence presented above (and Theorem \ref{thm:dweiksantambrogio}) holds also in an anisotropic version of the least gradient problem. Suppose that $\varphi$ is a strictly convex norm on $\mathbb{R}^2$; then, in light of the analysis performed in \cite{DS}, for strictly convex domains $\Omega$ we still have a one-to-one correspondence between gradients of $BV$ functions and vector-valued measures with zero divergence, and the reasoning presented in this Section still works in the anisotropic case. We summarise this in the following Remark.

\begin{remark}
Suppose that $\Omega \subset \mathbb{R}^2$ is an open bounded convex set. Suppose that $\varphi$ is a strictly convex norm on $\mathbb{R}^2$. Then, the infimal values in the anisotropic least gradient problem
\begin{equation}\tag{aLGP}\label{eq:anisotropiclgp}
\min \bigg\{ \int_\Omega \varphi(Du), \quad u \in BV(\Omega), \quad u|_{\partial\Omega} = g. \bigg\}
\end{equation}
and the anisotropic Beckmann problem
\begin{equation}\label{eq:anisotropicbeckmann}\tag{aBP}
  \min \bigg\{ \int_{\bar{\Omega}} \varphi(R_{- \frac{\pi}{2}} p)\,:\, p \in \mathcal{M}(\overline{\Omega}; \mathbb{R}^2),\,\,\mathrm{div}\,\, p = f \bigg\}
\end{equation}
coincide. Moreover, there is a correspondence between minimisers of the two problems as in Theorem \ref{thm:lgpbeckmannequivalence}. The anisotropic Beckmann problem is in turn equivalent to the Monge-Kantorovich problem with the cost given by the rotation norm of $\varphi$, i.e. $\varphi(R_{-\frac{\pi}{2}} \cdot)$, as in Theorem \ref{thm:beckmannkantorovichequivalence}.
\end{remark}

A weaker version of this statement holds also when $\varphi$ is not strictly convex; the correspondence between the anisotropic least gradient problem and the anisotropic Beckmann problem remains in place, but in Theorem \ref{thm:beckmannkantorovichequivalence} the last point is no longer true.
Note that whenever $\varphi$ is strictly convex, the transport rays are line segments and they may not intersect at an interior point; on the other hand, when $\varphi$ is not strictly convex, these properties may fail. Since any solution of the form $p_\gamma$ is concentrated on transport rays which are line segments, it is clear that in the non-strictly convex case it is possible that not every solution to the Beckmann problem is of the form $p_\gamma$. For these reasons, any regularity results require that the norm $\varphi$ is strictly convex. At the end of each Section, we will remark whether the results in this Section hold also for anisotropic norms and whether we need the norm to be strictly convex.

\section{Dual formulations}\label{sec:dualproblems}

In this Section, we extend the relationship between the least gradient problem \eqref{eq:leastgradientproblem} and the Kantorovich problem \eqref{eq:kantorovich} to their respective dual problems. Namely, we study the relationship between the maximisation problem \eqref{eq:dualtolgp} (see \cite{Gor2020TAMS,Mor})
\begin{equation*}
\sup \bigg\{ \int_{\partial\Omega} [\mathbf{z},\nu] \, g \, \mathrm{d}\mathcal{H}^{1}: \mathbf{z} \in \mathcal{Z} \bigg\},
\end{equation*}
where $g \in BV(\partial\Omega)$ and
\begin{equation*}
\mathcal{Z} = \bigg\{ \mathbf{z} \in L^\infty(\Omega;\mathbb{R}^2), \quad \mathrm{div}(\mathbf{z}) = 0, \quad \| \mathbf{z} \|_\infty \leq 1 \mbox{ a.e. in } \Omega \bigg\},
\end{equation*}
with the maximisation problem \eqref{eq:dualtokantorovich} (see \cite{San2015,Vil})
\begin{equation*}
\sup\bigg\{\int_{\overline{\Omega}} \phi\,\mathrm{d}(f^+ - f^-)\,:\,\phi \in  \mathrm{Lip}_1(\overline{\Omega})\bigg\}
\end{equation*}
whose solutions are known as Kantorovich potentials. Here, $f^\pm \in \mathcal{M}(\partial\Omega)$ with $f^+(\partial\Omega) = f^-(\partial\Omega)$. Moreover, the normal trace $[\mathbf{z},\nu]$ is understood in the weak sense, for details see \cite{Anz} or \cite{CF}. By a standard reasoning in duality theory, both problems admit solutions; this is proved for problem \eqref{eq:dualtolgp} in \cite{Mor} and for problem \eqref{eq:dualtokantorovich} for instance in \cite{San2015,Vil}. Since the infimal values in the primal problems are equal, it is clear that the supremal values in the dual problems coincide. The goal of this Section is to prove that from a solution of one problem we may recover a solution of the other problem and to study some consequences of this fact for the structure of solutions to the least gradient problem.

The following result is the main result in this Section and describes the relationship between the dual problems \eqref{eq:dualtolgp} and \eqref{eq:dualtokantorovich}.

\begin{theorem}\label{thm:dualproblems}
Let $\Omega\subset \mathbb{R}^2$ be an open bounded convex set. Then, the problems \eqref{eq:dualtokantorovich} and \eqref{eq:dualtolgp} are equivalent in the following sense: \\
(1) Their supremal values coincide, i.e. $\sup \eqref{eq:dualtokantorovich} = \sup \eqref{eq:dualtolgp}$; \\
(2) Given a maximiser $\phi \in \mathrm{Lip}_1(\overline{\Omega})$ of \eqref{eq:dualtokantorovich}, we can construct a maximiser $\mathbf{z} \in L^\infty(\Omega; \mathbb{R}^2)$ of \eqref{eq:dualtolgp}; moveover, $\mathbf{z} = R_{\frac{\pi}{2}} \nabla \phi$ in $\Omega$; \\
(3) Given a maximiser $\mathbf{z} \in L^\infty(\Omega;\mathbb{R}^2)$ of \eqref{eq:dualtolgp}, we may construct a maximiser $\phi \in \mathrm{Lip}_1(\overline{\Omega})$ of \eqref{eq:dualtokantorovich}; moreover, $\mathbf{z} = R_{\frac{\pi}{2}} \nabla \phi$ in $\Omega$.
\end{theorem}

Notice that the direction of the rotation is opposite to the direction of rotation in Theorem \ref{thm:lgpbeckmannequivalence}.

\begin{proof}
(1) This follows immediately from point (1) of Theorem \ref{thm:lgpbeckmannequivalence}, because
\begin{equation*}
\sup \eqref{eq:dualtolgp} = \inf \eqref{eq:leastgradientproblem} = \inf \eqref{eq:kantorovich} = \sup \eqref{eq:dualtokantorovich}.
\end{equation*}
(2) Suppose that $\phi \in \mathrm{Lip}_1(\overline{\Omega})$ is a maximiser in \eqref{eq:dualtokantorovich}. Take $\mathbf{z} = R_{\frac{\pi}{2}} \nabla (\phi|_\Omega) \in L^\infty(\Omega; \mathbb{R}^2)$. Then, $\mathbf{z}$ is admissible in \eqref{eq:dualtolgp}, since $\mathrm{div}(\mathbf{z}) = 0$ as distributions and $\| \mathbf{z} \|_\infty = \| \nabla \phi \|_\infty \leq 1$.

Note that since $\phi$ is Lipschitz and $g \in BV(\partial\Omega)$, we have that $\phi g \in BV(\partial\Omega)$ and it satisfies a mass balance condition
\begin{equation}\label{eq:lipbvmassbalance}
0 = \int_{\partial\Omega} \mathrm{d} \partial_\tau(\phi g) = \int_{\partial\Omega} \phi \, \mathrm{d}(\partial_\tau g) + \int_{\partial\Omega} (\partial_\tau \phi) \, g \, \mathrm{d}\mathcal{H}^1.
\end{equation}
Since $\mathbf{z} = R_{\frac{\pi}{2}} \nabla (\phi|_\Omega) = - R_{-\frac{\pi}{2}} \nabla (\phi|_\Omega)$, Proposition \ref{prop:recoveryinL1} implies that $[\mathbf{z},\nu] = - \partial_\tau \phi$. By equation \eqref{eq:lipbvmassbalance}, 
\begin{equation*}
\sup \eqref{eq:dualtolgp} = \sup \eqref{eq:dualtokantorovich} = \int_{\partial\Omega} \phi \, \mathrm{d}f = \int_{\partial\Omega} \phi \, d(\partial_\tau g) =  \int_{\partial\Omega} (- \partial_\tau \phi) \, g \, \mathrm{d}\mathcal{H}^1 =
\int_{\partial\Omega} [\mathbf{z},\nu] \, g \, \mathrm{d}\mathcal{H}^1,
\end{equation*}
hence $\mathbf{z}$ is a maximiser of \eqref{eq:dualtolgp}. \\

(3) Suppose that $\mathbf{z} \in L^\infty(\Omega; \mathbb{R}^2)$ is a maximiser in \eqref{eq:dualtolgp}. Since $\Omega$ is bounded, we also have $\mathbf{z} \in L^1(\Omega; \mathbb{R}^2)$, so by Proposition \ref{prop:recoveryinL1} there exists $u \in W^{1,1}(\Omega)$ such that $\mathbf{z} = R_{-\frac{\pi}{2}} \nabla u$. Since $\mathbf{z} \in L^\infty(\Omega;\mathbb{R}^2)$ with $\| \mathbf{z} \|_\infty \leq 1$, we also have $\nabla u \in L^\infty(\Omega;\mathbb{R}^2)$ with $\| \nabla u \|_\infty \leq 1$. Hence, $u \in W^{1,\infty}(\Omega)$, so it is $1$-Lipschitz in $\Omega$; in particular, it can be extended to a $1$-Lipschitz function on $\overline{\Omega}$ (we identify $u$ with its extension and write $u \in \mathrm{Lip}_1(\overline{\Omega})$), so it is admissible in problem \eqref{eq:dualtokantorovich}.

Moreover, Proposition \ref{prop:recoveryinL1} implies that $[\mathbf{z},\nu] = \partial_\tau u$. Again using equation \eqref{eq:lipbvmassbalance}, we get
\begin{equation*}
\sup \eqref{eq:dualtokantorovich} = \sup \eqref{eq:dualtolgp} = \int_{\partial\Omega} [\mathbf{z},\nu] \, g \, \mathrm{d}\mathcal{H}^1 = \int_{\partial\Omega} (\partial_\tau u) \, g \, \mathrm{d}\mathcal{H}^1 = - \int_{\partial\Omega} u \, d(\partial_\tau g) = \int_{\partial\Omega} (-u) \, \mathrm{d}f,
\end{equation*}
hence $\phi = -u$ is a maximiser of \eqref{eq:dualtokantorovich}. In particular, $\mathbf{z} = R_{-\frac{\pi}{2}} \nabla u = R_{\frac{\pi}{2}} \nabla \phi$.
\end{proof}

Hence, we may express the solution to the dual problem \eqref{eq:dualtolgp} to the least gradient problem via a Kantorovich potential of the corresponding optimal transport problem and vice versa. Of course, solutions to both problems are in general not unique; for instance, when $g \equiv c \in \mathbb{R}$, then any admissible vector field $\mathbf{z} \in \mathcal{Z}$ is a solution of \eqref{eq:dualtolgp}, and this corresponds to the fact that for $f = 0$ any $1$-Lipschitz function is a solution to problem \eqref{eq:dualtokantorovich}. In light of this, an important point is that any solution to problem \eqref{eq:dualtokantorovich} determines the frame of transport rays for all solutions to the Monge-Kantorovich problem, see for instance \cite{San2015}; similarly, any solution to problem \eqref{eq:dualtolgp} determines the frame of level sets to all solutions to the least gradient problem in the following sense (see \cite{Mor}): the vector field $\overline{\mathbf{z}} = -\mathbf{z}$ satisfies the Euler-Lagrange equations for the least gradient problem introduced in \cite{MRL}, in particular it is divergence-free
\begin{equation*}
(\overline{\mathbf{z}}, Du) = |Du| \qquad \mbox{ as measures in } \Omega,
\end{equation*}
where $(\overline{\mathbf{z}}, Du)$ is the Anzelotti pairing defined in \cite{Anz}, which is a generalisation of the pointwise product $\overline{\mathbf{z}} \cdot \nabla u$ to the case when $u \in BV(\Omega)$ and $\overline{\mathbf{z}}$ is a bounded vector field with divergence in $L^N(\Omega)$. Hence, if $\phi$ is a Kantorovich potential, then $\overline{\mathbf{z}} = R_{-\frac{\pi}{2}} \nabla \phi$. We illustrate this by revisiting a classical example, attributed to John Brothers and appearing for instance in \cite{MRL,SZ}.

\begin{example}\label{ex:brothers}
Suppose that $\Omega = B(0,1) \subset \mathbb{R}^2$. We will first give an example with continuous boundary data and then modify it to an example with discontinuous boundary data. Set $g_1(x,y) = x^2 - y^2$; since the boundary datum is continuous, the solution to the least gradient problem is unique (see \cite{SWZ}) and we may easily check that it is given by the formula
\begin{equation*}
u(x,y) = \threepartdef{2x^2 - 1}{\mathrm{if} \, |x| > \frac{\sqrt{2}}{2},}{1 - 2y^2}{\mathrm{if} \, |y| > \frac{\sqrt{2}}{2},}{0}{\mathrm{otherwise.}}
\end{equation*}
Since $u$ is Lipschitz, the condition $(\overline{\mathbf{z}},Du) = |Du|$ reduces to $\overline{\mathbf{z}} = \frac{\nabla u}{|\nabla u|}$ almost everywhere on the support of $\nabla u$. Hence, up to choosing a representative, we have that $\overline{\mathbf{z}} = (\mathrm{sgn}(x),0)$ if $|x| > \frac{\sqrt{2}}{2}$, $\overline{\mathbf{z}} = (0,-\mathrm{sgn}(y))$ if $|y| > \frac{\sqrt{2}}{2}$, and $\overline{\mathbf{z}}$ is not uniquely defined on the square $|x|,|y| < \frac{\sqrt{2}}{2}$. For instance, we may take
\begin{equation}
\overline{\mathbf{z}}(x,y) = \fourpartdef{(1,0)}{\mathrm{if} \, -x < y < x,}{(0,-1)}{\mathrm{if} \, y > x \, \mathrm{and} \, y > -x,}{(-1,0)}{\mathrm{if} \, x < y < -x,}{(0,1)}{\mathrm{if} \, y < x \, \mathrm{and} \, y < -x.}
\end{equation}
Now, we look at the problem from the optimal transport angle. Using angular coordinates on $\partial\Omega$, we see that $g_1(\theta) = \cos(2\theta)$. Clearly, $g_1 \in C(\partial\Omega) \cap BV(\partial\Omega)$, and its tangential derivative is given by the formula $f(\theta) = - 2 \sin(2\theta)$. By Theorem \ref{thm:dualproblems}, a Kantorovich potential corresponding to $f$ can be obtained by taking $\nabla \phi_1 = R_{-\frac{\pi}{2}} \mathbf{z} = R_{\frac{\pi}{2}} \overline{\mathbf{z}}$, so up to an additive constant we get
\begin{equation}
\phi_1(x,y) = \fourpartdef{-y + \frac{\sqrt{2}}{2}}{\mathrm{if} \, -x < y < x,}{-x + \frac{\sqrt{2}}{2}}{\mathrm{if} \, y > x \, \mathrm{and} \, y > -x,}{y + \frac{\sqrt{2}}{2}}{\mathrm{if} \, x < y < -x,}{x + \frac{\sqrt{2}}{2}}{\mathrm{if} \, y < x \, \mathrm{and} \, y < -x.}
\end{equation}
We chose the additive constant in the standard way, so that the minimal value of $\phi$ equals zero. The situation is presented on the left hand side of Figure \ref{fig:pbpotencjal}. However, it is not difficult to construct another Kantorovich potential. Notice that the function $\phi_2$ such that its level sets are horizontal line segments for $|x| > \frac{\sqrt{2}}{2}$ (and vertical line segments for $|y| > \frac{\sqrt{2}}{2}$) and in the square $|x|,|y| < \frac{\sqrt{2}}{2}$ each level set is a part of a circle with centre at a point $(\pm \frac{\sqrt{2}}{2}, \pm \frac{\sqrt{2}}{2})$, has the same boundary values as $\phi_1$ and that it is $1$-Lipschitz. Hence, it is admissible in problem \eqref{eq:dualtokantorovich} and it is optimal by optimality of $\phi_1$. The situation is presented on the right hand side of Figure \ref{fig:pbpotencjal}.

\begin{figure}[h]
    \includegraphics[scale=0.1]{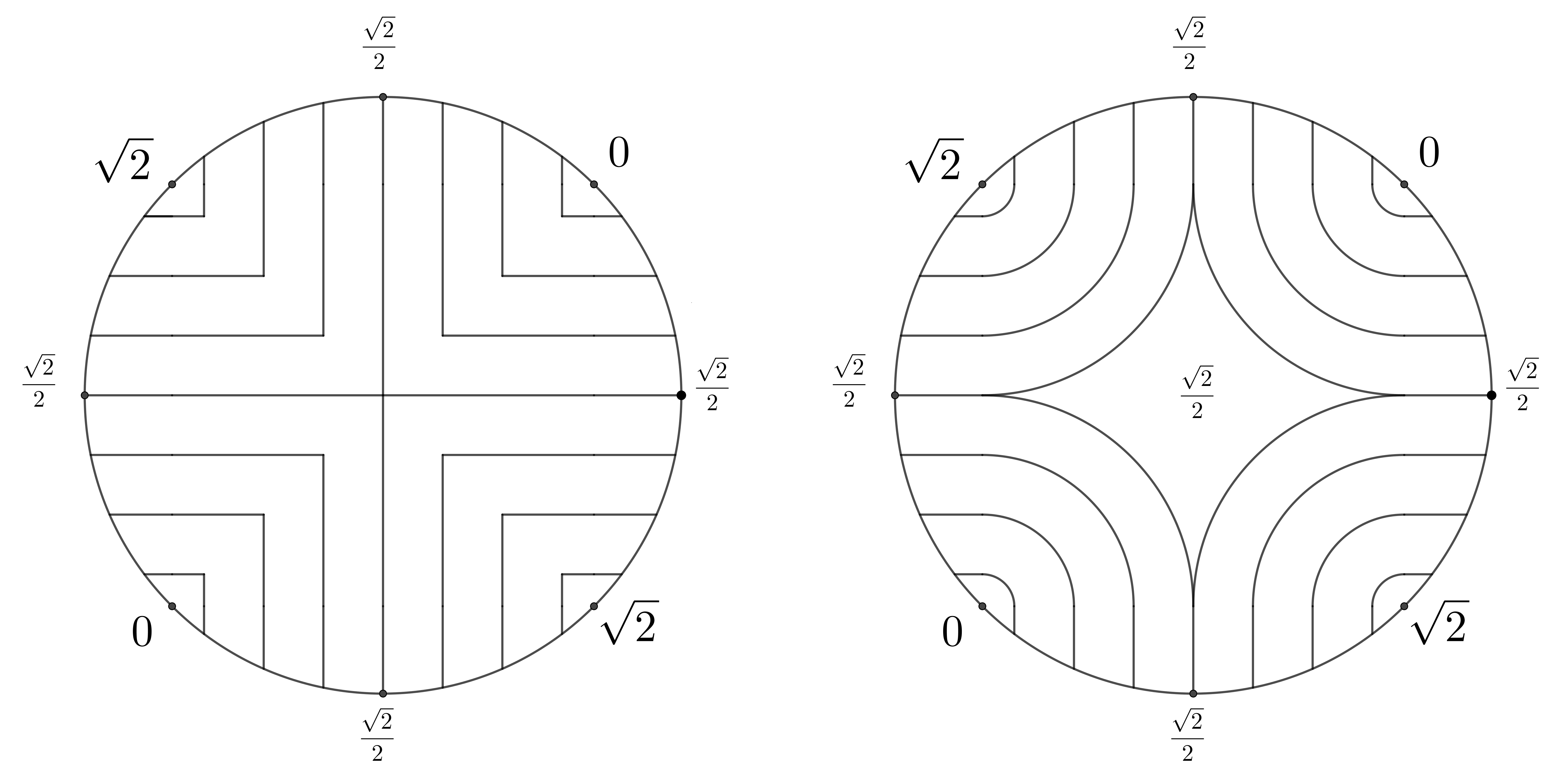}
    \caption{Two possible Kantorovich potentials}
    \label{fig:pbpotencjal}
\end{figure}

Finally, when we modify the boundary datum a bit, so that
\begin{equation*}
g_2(x,y) = \twopartdef{x^2 - y^2 + 1}{\mathrm{if} \, |x| > \frac{\sqrt{2}}{2},}{x^2 - y^2 - 1}{\mathrm{if} \, |y| > \frac{\sqrt{2}}{2},}
\end{equation*}
then the functions of the form
\begin{equation*}
u_\lambda(x,y) = \threepartdef{2x^2}{\mathrm{if} \, |x| > \frac{\sqrt{2}}{2},}{- 2y^2}{\mathrm{if} \, |y| > \frac{\sqrt{2}}{2},}{\lambda}{\mathrm{otherwise}}
\end{equation*}
where $\lambda \in [-1,1]$ are all possible solutions to the least gradient problem (see \cite{Gor2018JMAA,MRL}). The situation is presented in Figure \ref{fig:pbnieciagly}. It was shown in \cite{MRL} that the vector fields $\overline{\mathbf{z}}$ which solve the Euler-Lagrange equations (so $- \overline{\mathbf{z}}$ is a solution of \eqref{eq:dualtolgp}) for $g_1$ and $g_2$ are the same. Hence, by Theorem \ref{thm:dualproblems} also the Kantorovich potentials corresponding to $f_1 = \partial_\tau g_1$ and $f_2 = \partial_\tau g_2$ are the same. Notice that even though the solution to problem \eqref{eq:leastgradientproblem} is not unique, all the solutions share the same frame of superlevel sets, and it can be described using Theorem \ref{thm:dualproblems} in terms of any Kantorovich potential: for any solution $u$ of \eqref{eq:leastgradientproblem} and any Kantorovich potential $\phi$, whenever a level line of $u$ and a level line of $\phi$ intersect, they make a right angle.

\begin{figure}[h]
    \includegraphics[scale=0.3]{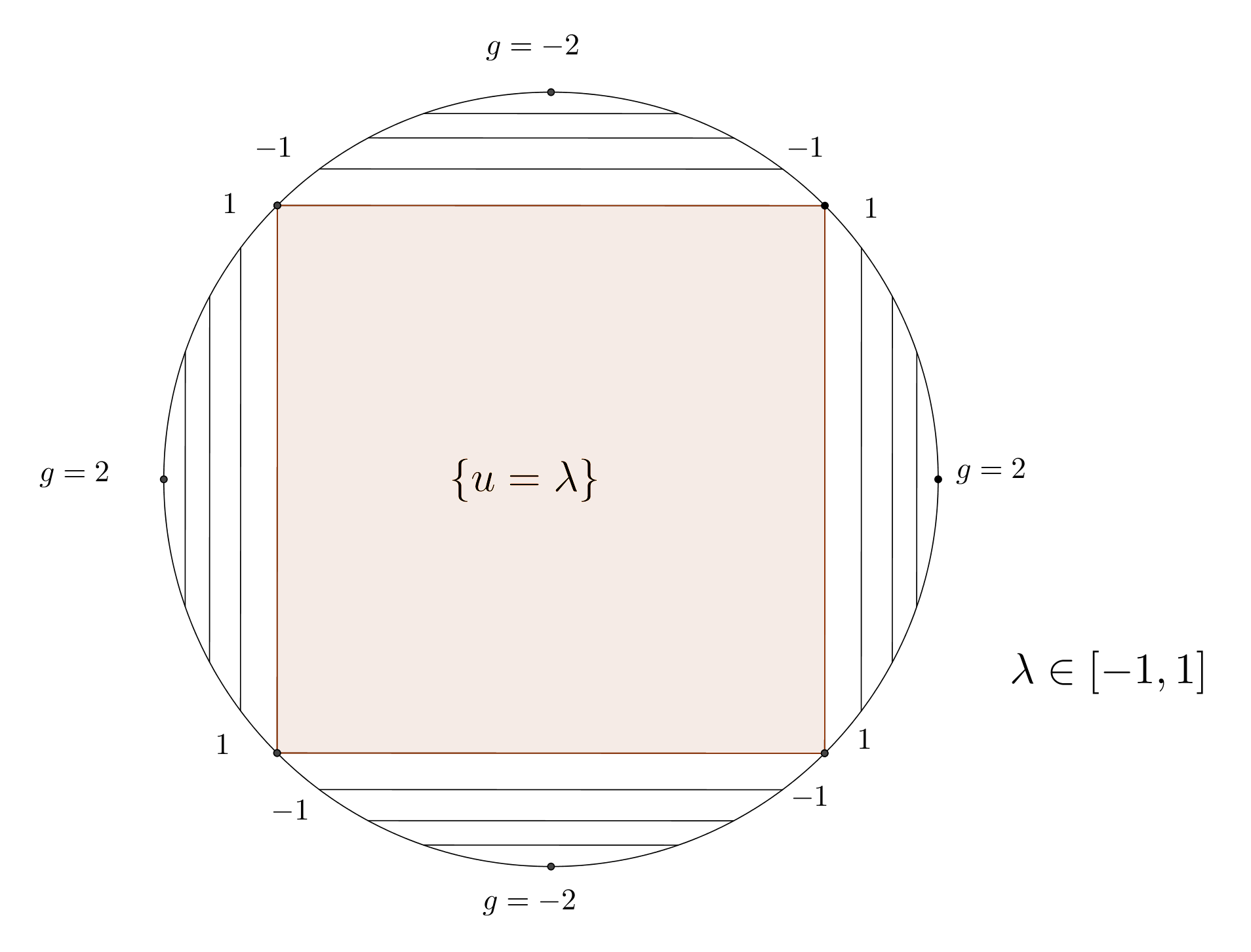}
    \caption{Multiple solutions to the least gradient problem}
    \label{fig:pbnieciagly}
\end{figure}
\end{example}

In fact, given boundary data $g \in BV(\partial\Omega)$, it may be easier to construct Kantorovich potentials and from it reconstruct the vector field $\mathbf{z}$ which is the solution of problem \eqref{eq:dualtolgp}, as we saw in the Example above. Furthermore, notice that on the set on which the solution $u$ to the least gradient problem was not constant, level lines of the Kantorovich potentials were perpendicular to level lines of $u$; indeed, Theorem \ref{thm:dualproblems} can be understood as an informal link between solutions of the primal problem \eqref{eq:leastgradientproblem} and \ref{eq:kantorovich}; namely, suppose that $l_t$ is a connected component of the boundary of a superlevel set $\{ u > t \}$ of $u$, a solution to \eqref{eq:leastgradientproblem}; by \cite[Theorem 1]{BGG}, it is a line segment. If $u$ is regular enough, then the condition $(\overline{\mathbf{z}},Du) = |Du|$ implies that $\mathbf{z}$ is a unit vector perpendicular to $l_t$, at least for almost all $t$. The gradient of the corresponding Kantorovich potential $\phi$ is in turn a unit vector perpendicular to $\overline{\mathbf{z}}$, hence it is parallel to $l_t$; but since the gradient of the Kantorovich potential is well-defined and of length one, this means that $l_t$ is a transport ray. Hence, level sets of a solution to \eqref{eq:leastgradientproblem} have the interpretation of transport rays.

Finally, let us note that the equivalence given in Theorem \ref{thm:dualproblems} holds also in an anisotropic version of the least gradient problem. Strict convexity of the norm is not required; we only use Proposition \ref{prop:recoveryinL1}, which is a pointwise result regardless of the norm used. Hence, we get the following result.

\begin{remark}
Suppose that $\Omega \subset \mathbb{R}^2$ is an open bounded convex set. Suppose that $\varphi$ is any norm on $\mathbb{R}^2$. Then, the supremal values in the dual of the anisotropic least gradient problem
\begin{equation*}
\sup \bigg\{ \int_{\partial\Omega} [\mathbf{z},\nu] \, g \, \mathrm{d}\mathcal{H}^{1}: \mathbf{z} \in \mathcal{Z} \bigg\},
\end{equation*}
where
\begin{equation*}
\mathcal{Z} = \bigg\{ \mathbf{z} \in L^\infty(\Omega;\mathbb{R}^2), \quad \mathrm{div}(\mathbf{z}) = 0, \quad \varphi^0(\mathbf{z}(x)) \leq 1 \mbox{ a.e. in } \Omega \bigg\},
\end{equation*}
and the dual of the Kantorovich problem with the cost given by the rotation norm of $\varphi$, i.e. $\varphi(R_{-\frac{\pi}{2}} \cdot)$
\begin{equation*}
\sup\bigg\{\int_{\overline{\Omega}} \phi\,\mathrm{d}(f^+ - f^-): \quad \phi \in  \mathrm{Lip}(\overline{\Omega}), \quad \varphi^0(R_{-\frac{\pi}{2}}  \nabla \phi) \leq 1 \mbox{ a.e. in } \Omega \bigg\}
\end{equation*}
coincide. Here, $\varphi^0$ denotes the polar norm of $\varphi$. Moreover, there is a correspondence between minimisers of the two problems as in Theorem \ref{thm:dualproblems}.
\end{remark}

\section{Applications to regularity}\label{sec:SBVregularity}

In this Section, we inspect the regularity of solutions to the least gradient problem. The first results in this direction obtained using optimal transport methods were proved for uniformly convex domains in \cite{DS}. There, assuming $W^{1,p}(\partial\Omega)$ regularity of the boundary datum with $p \leq 2$, the authors proved that the (unique) solution lies in $W^{1,p}(\Omega)$. This was achieved in the following way: $W^{1,p}$ regularity of the boundary datum in problem \eqref{eq:leastgradientproblem} corresponds to $L^p$ regularity of the boundary datum in problem \eqref{eq:kantorovich}. Then, the authors prove $L^p$ estimates on the transport density $\sigma_\gamma$; using the relation $|Du| = |p| = |\sigma_\gamma|$ between solutions of problems \eqref{eq:leastgradientproblem}, \eqref{eq:kantorovich} and \eqref{eq:beckmannproblem}, we see that this corresponds to $W^{1,p}$ regularity of the solution to the least gradient problem.

The goal of this Section is to study the case when the boundary datum is less regular than in the situation studied in \cite{DS}. We start by proving the main result in this Section, namely Theorem \ref{thm:sbvregularity}, which says that on uniformly convex domains if the boundary datum lies in $SBV(\partial\Omega)$, then any solution lies in $SBV(\Omega)$. Note that the boundary data are of bounded variation, but they are no longer continuous; hence, there exists solutions to problem \eqref{eq:leastgradientproblem} and they may fail to be unique (as we saw in Example \ref{ex:brothers}). In the corresponding Kantorovich problem, it means that the optimal transport plan may fail to be unique or induced by a map. Nonetheless, the result holds for every minimiser. Then, we focus on some consequences of the proof, in particular on structure results in the case when $g$ has only finitely many discontinuities. 

Since at some point we will rely on results in \cite{DS}, we will assume that $\Omega$ is uniformly convex. The definition used there is a bit more general than the standard one used in the literature (i.e. $\partial\Omega$ is smooth and the mean curvature is bounded from below by a positive constant); namely, the authors of \cite{DS} assume that there exists $R > 0$ such that for every $x \in \partial\Omega$ and every inner unit vector $\mathbf{n}$ we have $\Omega \subset B(x+R\mathbf{n},R)$. For smooth domains, this corresponds to the fact that all principal curvatures are larger than $\frac{1}{R}$, so the two definitions coincide. We will adopt this definition of uniform convexity in order not to require $\partial\Omega$ to be smooth.

\begin{theorem}\label{thm:sbvregularity}
Suppose that $\Omega \subset \mathbb{R}^2$ is uniformly convex. Let $g \in SBV(\partial\Omega)$. If $u \in BV(\Omega)$ is a solution to problem \eqref{eq:leastgradientproblem} with boundary data $g$, then $u \in SBV(\Omega)$.
\end{theorem}

In the course of the proof, supposing that $\mu$ and $\nu$ are two positive measures, we say that $\mu \leq \nu$ if for all Borel sets $B$ we have $\mu(B) \leq \nu(B)$. In particular, in this case $\mu$ is absolutely continuous with respect to $\nu$.

\begin{proof}
Denote $f = \partial_\tau g$. Since $g \in SBV(\Omega)$, we have that $f = f_{ac} + f_{at}$, where $f_{ac}$ is absolutely continuous and $f_{at}$ is atomic; there is no Cantor part. We decompose $f$ into a positive and negative part, namely $f = f^+ - f^-$ and recall that the least gradient problem corresponds to the optimal transport problem between $f^+$ and $f^-$; to every $u \in BV(\Omega)$ which is a solution of the least gradient problem, there exists an optimal transport plan $\overline{\gamma} \in \mathcal{M}^+(\overline{\Omega} \times \overline{\Omega})$ with marginals $f^+$ and $f^-$.

We will decompose the Monge-Kantorovich problem
\begin{equation}\label{eq:probleminproofofsbv}
    \min\bigg\{\int_{\overline{\Omega} \times \overline{\Omega}} |x-y|\,\mathrm{d}\gamma\,:\,\gamma \in  \mathcal{M}^+(\overline{\Omega} \times \overline{\Omega}),\,(\Pi_x)_{\#}\gamma=f^+ \,\,\mathrm{and}\,\,(\Pi_y)_{\#} \gamma =f^-\bigg\}
\end{equation}
into several problems of the same type. Then, we will compute the transport densities and prove that they are either absolutely continuous or concentrated on a set of Hausdorff dimension one. This will imply that $Du$ contains no Cantor part, so $u \in SBV(\Omega)$.

{\bf Step 1.} Let us introduce the following notation: denote by $D^+$ the (at most countable) set of atoms of $f^+$ and by $D^-$ the (at most countable) set of atoms of $f^-$. Since $f^+$ and $f^-$ have no common mass, these sets are disjoint. Moreover, for $x \in \overline{\Omega}$, denote by $\Delta_x$ the set of all points from transport rays passing through the point $x$. It is clear that $\Delta_x$ is closed: if we denote
\begin{equation*}
h_x(y) = |\phi(y) - \phi(x)| - |x-y|,
\end{equation*}
then since $\phi$ is $1$-Lipschitz, we have that $h \leq 0$ and $\Delta_x = h_x^{-1}(0)$, hence it is a closed set.

Now, we will separate $\overline{\Omega}$ into a few Borel subsets. Denote
\begin{equation*}
A_1 := \bigcup_{p \in D^+, \, q \in D^-} [p, q]. 
\end{equation*}
Clearly, $A_1$ is a Borel set and its Hausdorff dimension equals one (although its $\mathcal{H}^1$ measure may be infinite). In particular, all transport rays with both endpoints in the atoms of $f$ lie in $A_1$; in other words, $A_1$ is the set on which the transport between atomic parts of $f^+$ and $f^-$ takes place. We will later see that the jump set of $u$ will be a subset of this set.

Denote
\begin{equation*}
A_2 := \bigg( \bigg( \bigcup_{p \in D^+} \Delta_{p} \bigg) \setminus A_1 \bigg) \cup D^+.
\end{equation*}
Since $\Delta_x$ is closed for every $x \in \overline{\Omega}$, $A_2$ is a Borel set. In particular, all transport rays with one endpoint in an atom of $f^+$ and the other endpoint not in an atom of $f^-$ lie in $A_2$. In other words, $A_2$ is the set on which the transport between the atomic part of $f^+$ and the absolutely continuous part of $f^-$ takes place.

Similarly, denote
\begin{equation*}
A_3 := \bigg( \bigg( \bigcup_{q \in D^-} \Delta_{q} \bigg) \setminus A_1 \bigg) \cup D^-.
\end{equation*}
Again, $A_3$ is a Borel set and all transport rays with one endpoint in an atom of $f^-$ and the other endpoint not in an atom of $f^+$ lie in $A_3$. In other words, $A_3$ is the set on which the transport between the absolutely continuous part of $f^+$ and the atomic part of $f^-$ takes place.

Finally, denote
\begin{equation*}
A_4 := \bigg( \Omega \setminus \bigg( A_1 \cup A_2 \cup A_3 \bigg) \bigg) \cup \bigg( \partial\Omega \setminus \bigg( D^+ \cup D^- \bigg) \bigg).
\end{equation*}
Clearly, $A_4$ is a Borel set and all transport rays with both endpoints not in an atom of $f$ lie in $A_4$. In other words, $A_4$ is the set on which the transport between the absolutely continuous parts of $f^+$ and $f^-$ takes place.

Note that the union of the sets $A_n$ (for $n=1,2,3,4$) equals $\overline{\Omega}$. These sets are not disjoint, but we have some control over the intersections: in particular, $A_1 \cap A_2 = D^+$ and $A_1 \cap A_3 = D^-$. Moreover, since the set of points which belong to at least two transport rays is countable, the other intersections are at most countable and do not contain any atom of $f$.

{\bf Step 2.} First, notice that the optimal transport plan $\overline{\gamma}$ is concentrated on $\partial\Omega \times \partial\Omega$. Namely, since $(\Pi_x)_\# \overline{\gamma} = f^+$, we have $\overline{\gamma}(\Omega \times \overline{\Omega}) = (\Pi_x)_\# \overline{\gamma} (\Omega) = f^+(\Omega) = 0$; similarly, since $(\Pi_y)_\# \overline{\gamma} = f^-$, we have $\overline{\gamma}(\overline{\Omega} \times \Omega) = (\Pi_y)_\# \overline{\gamma}(\Omega) = f^-(\Omega) = 0$. Hence, we have $\overline{\gamma} = \overline{\gamma}|_{\partial\Omega \times \partial\Omega}$. Now, we will separate $\partial\Omega \times \partial\Omega$ into a few Borel subsets and study restrictions of $\overline{\gamma}$ to these subsets. Recall that two transport rays may only intersect at their endpoints and whenever $(x_0,y_0)$ belongs to the support of an optimal transport plan, then $x_0$ and $y_0$ must belong to a common transport ray.

We make the following decomposition of $\partial\Omega \times \partial\Omega$: for $n = 1,2,3$ we set
\begin{equation*}
B_n = (A_n \cap \partial\Omega) \times (A_n \cap \partial\Omega)
\end{equation*}
and we set
\begin{equation*}
B_4 = \{ (x,y) \in \partial\Omega \times \partial\Omega: x,y \notin B_1 \cup B_2 \cup B_3 \}.
\end{equation*}
By definition, we have $\bigcup_{n=1}^4 B_n = \partial\Omega \times \partial\Omega$. As was the case with the sets $A_n$, again the sets $B_n$ are not disjoint, but we have some control over the intersections; namely, for $n=1,2,3$ we have $B_n \cap B_4 = \emptyset$, $B_1 \cap B_2 = \bigcup_{p \in D^+} \{ (p, p) \}$, $B_1 \cap B_3 = \bigcup_{q \in D^-} \{ (q, q) \}$, and $B_2 \cap B_3$ is countable and each of the points is of the form $\{ (x,x) \}$, where $x \in A_2 \cap A_3 \cap \partial\Omega$.

Then, we decompose the optimal plan $\overline{\gamma}$ as follows; for $n=1,2,3,4$, we set $\gamma_n = \overline{\gamma}|_{B_n}$. Since all the sets $B_n$ are Borel, all the measures $\gamma_n$ are positive Radon measures. It is clear that $\gamma_n \leq \overline{\gamma}$ as measures. Moreover, notice that an optimal transport plan gives no mass to any single point in the diagonal $(x,x)$ - otherwise $x$ would be an atom of both $f^+$ and $f^-$. Hence, since all the intersections $B_m \cap B_n$ for $m,n = 1,2,3,4$ are at most countable unions of such points, we have that $\sum_{n=1}^4 \gamma_n = \overline{\gamma}|_{\partial\Omega \times \partial\Omega} = \overline{\gamma}$.

{\bf Step 3.} Now, we construct the auxiliary problems. First, let us notice that since for all $n = 1,2,3,4$ we have $\gamma_n \leq \overline{\gamma}$ as measures, then
\begin{equation}\label{eq:estimateonmarginals1}
(\Pi_x)_{\#} \gamma_n \leq (\Pi_x)_{\#} \overline{\gamma} = f^+
\end{equation}
and 
\begin{equation}\label{eq:estimateonmarginals2}
(\Pi_y)_{\#} \gamma_n \leq (\Pi_y)_{\#} \overline{\gamma} = f^-.
\end{equation}
Then, let us notice that since $\sum_{n=1}^4 \gamma_n = \overline{\gamma}$, we also have
\begin{equation*}
\sum_{n=1}^4 (\Pi_x)_{\#} \gamma_n = f^+ \quad \mathrm{and} \quad \sum_{n=1}^4 (\Pi_y)_{\#} \gamma_n = f^-.
\end{equation*}
Now, notice that each $\gamma_n$ solves the problem
\begin{equation}
\min\bigg\{\int_{\overline{\Omega} \times \overline{\Omega}} |x-y|\,\mathrm{d}\gamma\,:\,\gamma \in  \mathcal{M}^+(\overline{\Omega} \times \overline{\Omega}),\,(\Pi_x)_{\#}\gamma = (\Pi_x)_{\#} \gamma_n \,\,\mathrm{and}\,\,(\Pi_y)_{\#} \gamma = (\Pi_y)_{\#} \gamma_n \bigg\}.
\end{equation}
By standard optimal transport theory, this problem has a solution $\gamma_n'$. If $\gamma_n$ is not a solution, then we may replace $\gamma_n$ by $\gamma_n'$, and then
\begin{equation*}
\gamma' = \sum_{i=1}^4 \gamma_n'   
\end{equation*}
is admissible in the Kantorovich problem \eqref{eq:probleminproofofsbv} and satisfies
\begin{equation*}
\int_{\overline{\Omega} \times \overline{\Omega}} |x-y| \, \mathrm{d} \gamma' \leq \sum_{n=1}^4 \int_{\overline{\Omega} \times \overline{\Omega}} |x-y| \, \mathrm{d} \gamma_n' < \sum_{n=1}^4 \int_{\overline{\Omega} \times \overline{\Omega}} |x-y| \, \mathrm{d} \gamma_n = \int_{\overline{\Omega} \times \overline{\Omega}} |x-y| \, \mathrm{d} \overline{\gamma},
\end{equation*}
hence $\overline{\gamma}$ was not an optimal transport plan, contradiction.

{\bf Step 4.} Since $\overline{\gamma} = \sum_{n=1}^4 \gamma_i$, the transport densities $\sigma_{\overline{\gamma}}$ of $\overline{\gamma}$ and $\sigma_{\gamma_n}$ of $\gamma_n$ defined by equation \eqref{eq:definitionofsigma} satisfy
\begin{equation*}
\sigma_{\overline{\gamma}} = \sum_{n=1}^4 \sigma_{\gamma_n}.
\end{equation*}
We will study separately the transport densities $\sigma_{\gamma_n}$. First, we will prove that for $n=2,3,4$ we have $\sigma_{\gamma_n} \in L^1(\Omega)$; then, we will study in more detail the structure of $\sigma_{\gamma_1}$.

First, notice that by equation \eqref{eq:estimateonmarginals1}, each of the measures $(\Pi_x)_\# \gamma_n$ is a sum of an absolutely continuous measure and an atomic measure, and any atom of $(\Pi_x)_\# \gamma_n$ is also an atom of $f^+$ (a point in $D^+$). Similarly, by \eqref{eq:estimateonmarginals2}, also each $(\Pi_y)_\# \gamma_n$ is a sum of an absolutely continuous measure and an atomic measure, and any atom of $(\Pi_y)_\# \gamma_n$ is also an atom of $f^-$ (a point in $D^-$).

We will prove that $(\Pi_y)_{\#} \gamma_2$ is absolutely continuous. By the previous paragraph, it suffices to show that no point $q \in D^-$ is its atom. We write
\begin{equation*}
(\Pi_y)_\# \gamma_2(\{ q \}) = \gamma_2(\overline{\Omega} \times \{ q \}) = \gamma_2(\Delta_{q} \times \{ q \}) = \overline{\gamma}((\Delta_{q} \times \{ q \}) \cap (A_2 \times A_2)) = \gamma(\emptyset) = 0,
\end{equation*}
where the first equality follows from the definition of the marginal, the second and third ones from the definition of $\gamma_2$, and the fourth one from the fact that $q \notin A_2$ for all $q \in D^-$. Hence, $(\Pi_y)_{\#} \gamma_2$ is absolutely continuous, so Theorem \ref{thm:dweiksantambrogio} implies that $\sigma_{\gamma_2} \in L^1(\Omega)$.

Similarly, we see that $(\Pi_x)_{\#} \gamma_3$ is absolutely continuous: again, it suffices to show that no point $p \in D^+$ is its atom, but then $p \notin A_3$ implies
\begin{equation*}
(\Pi_x)_\# \gamma_3(\{ p \}) = \gamma_3(\{ p \} \times \overline{\Omega}) = \gamma_3(\{ p \} \times \Delta_{p}) = \overline{\gamma}((\{ p \} \times \Delta_{p}) \cap (A_3 \times A_3)) = \overline{\gamma}(\emptyset) = 0,
\end{equation*}
hence $(\Pi_x)_{\#} \gamma_3$ is absolutely continuous, so Theorem \ref{thm:dweiksantambrogio} implies that $\sigma_{\gamma_3} \in L^1(\Omega)$. A minor variation of the above arguments shows that both $(\Pi_x)_{\#} \gamma_4$ and $(\Pi_y)_\# \gamma_4$ are absolutely continuous, so also $\sigma_{\gamma_4} \in L^1(\Omega)$.

Finally, we study the transport density $\sigma_{\gamma_1}$. First, notice that by property \eqref{eq:definitionofsigmaversiontwo} we have
\begin{equation*}
\sigma_{\gamma_1}(\overline{\Omega} \setminus A_1) = \int_{\overline{\Omega} \times \overline{\Omega}} \mathcal{H}^1([x,y] \setminus A_1) \, \mathrm{d} \gamma_1(x,y) = \int_{(A_1 \cap \partial\Omega) \times (A_1 \cap \partial\Omega)} \mathcal{H}^1([x,y] \setminus A_1) \, \mathrm{d} \gamma(x,y).
\end{equation*}
Recall that $A_1$ is an at most countable union of line segments connecting two points in $\partial\Omega$, one of which belongs to $D^+$ and the other one to $D^-$. Notice that $A_1 \cap \partial\Omega = D^+ \cup D^-$, in particular it is countable. Then, for any $x,y \in A_1 \cap \partial\Omega$, we have two possibilities: if $x, y \in D^+$ (or $x,y \in D^-$), then $\gamma(\{ x,y \}) = 0$. On the other hand, if $x \in D^+$ and $y \in D^-$ (or $x \in D^-$ and $y \in D^+$), then $[x,y] \subset A_1$, so $\mathcal{H}^1([x,y] \setminus A_1) = 0$. In either case, we have $\mathcal{H}^1([x,y] \setminus A_1) \gamma(\{ x,y \}) = 0$, so
\begin{equation*}
\sigma_{\gamma_1}(\overline{\Omega} \setminus A_1) = 0.
\end{equation*}
By equation \eqref{eq:definitionofsigmaversiontwo}, $\sigma_{\gamma_1}$ is absolutely continuous with respect to $\mathcal{H}^1$. Hence, actually $\sigma_{\gamma_1}$ a finite positive Radon measure which is absolutely continuous with respect to $\mathcal{H}^1|_{A_1}$.

{\bf Step 5.} Finally, since $|Du| = |p| = \sigma_{\overline{\gamma}}$, where $p$ is the solution to the Beckmann problem 
corresponding to $u$ and $\overline{\gamma}$ is the optimal transport plan corresponding to $p$, we have that
\begin{equation*}
|Du| = \sigma_{\overline{\gamma}} = \sigma_{\gamma_1} + (\sigma_{\gamma_2} + \sigma_{\gamma_3} + \sigma_{\gamma_4}),
\end{equation*}
hence $|Du|$ is a sum of an absolutely continuous measure $(\sigma_{\gamma_2} + \sigma_{\gamma_3} + \sigma_{\gamma_4})$ and a measure concentrated on a set of Hausdorff dimension one $\sigma_{\gamma_1}$, so $u \in SBV(\Omega)$.
\end{proof}

Actually, the proof of the Theorem works in a slightly general setting. In higher dimensions, with $\mathbb{R}^d$ equipped with the Euclidean norm, transport rays are still line segments, and the same proof yields the following result.

\begin{corollary}\label{cor:transportregularity}
Suppose that $\Omega \subset \mathbb{R}^d$ is uniformly convex. Suppose that $f \in \mathcal{M}(\partial\Omega)$ may be decomposed in the following way: $f = f_{ac} + f_{at}$, where $f_{ac} \in L^1(\partial\Omega)$ and $f_{at}$ is atomic. Then, when $\gamma$ is an optimal transport plan corresponding to the optimal transport problem between $f^+$ and $f^-$, we have
\begin{equation*}
\sigma_\gamma = \sigma_{ac} + \sigma_{at},
\end{equation*}
where $\sigma_{ac} \in L^1(\Omega)$ and $\sigma_{at}$ is concentrated on a set of Hausdorff dimension one.
\end{corollary}

However, this is an a bit unusual setting to study the optimal transport problem, and in higher dimensions we do not have the correspondence between the least gradient  problem and the optimal transport problem, so we prefer to state Theorem \ref{thm:sbvregularity} in its current form.

In fact, a careful inspection of the proof of Theorem \ref{thm:sbvregularity} enables us to study in more detail the regularity and structure of solutions in the case when $g$ has only finitely many discontinuities. We will give some of the properties obtained in this way below.

\begin{corollary}
Suppose that $\Omega \subset \mathbb{R}^2$ is uniformly convex. Let $g \in BV(\partial\Omega)$ and suppose that the set $D$ of its discontinuity points is finite. Suppose further that $g \in W^{1,p}(\partial\Omega \setminus D)$ with $p \in (1,2)$. Set $J = \bigcup_{p,q \in D} [p,q]$; then, if $u \in BV(\Omega)$ is a solution to problem \eqref{eq:leastgradientproblem} with boundary data $g$, then $u \in W^{1,p}(\Omega \setminus J)$. \qed
\end{corollary}

The requirement that $D$ is finite is required to be able to formulate the result in this setting, so that $J$ is a relatively closed set and $u$ is actually a function in a broken Sobolev space: by writing $u \in W^{1,p}(\Omega \setminus J)$, we mean that $\Omega \setminus J$ is an open set with finitely many connected components $(\Omega_i)_{i=1}^m$ and $u \in W^{1,p}(\Omega_i)$ for all $i = 1,...,m$. We use the same notation on $\partial\Omega$. Without this assumption, the statement of Corollary \ref{cor:transportregularity} with $\sigma_{ac} \in L^p(\Omega)$ is still valid.

The situation is different for $p > 2$. The results in \cite{DS} in this case require higher regularity of both $f^+$ and $f^-$; assuming regularity of one of them does not suffice. Indeed, if $u \in W^{1,p}(\Omega)$, then actually $u$ is H\"older continuous on $\overline{\Omega}$; but this in turn implies that $f$ is atomless, a contradiction. Hence, in the second part of this Section we will instead focus on local regularity of solutions to \eqref{eq:leastgradientproblem}, the main result being that any solution is Lipschitz in a neigbourhood of a generic point in $\Omega$. We start by proving the following Proposition. 

\begin{proposition}\label{prop:locallyfinite}
Suppose that $\Omega \subset \mathbb{R}^d$ is uniformly convex, where $d \geq 2$. Let $f^+ \in L^\infty(\partial\Omega)$ and take any $f^- \in \mathcal{M}(\partial\Omega)$. Then, if $\sigma$ is the (uniquely defined) transport density between $f^+$ and $f^-$, for all $\varepsilon > 0$ we have $\sigma \in L^\infty(\Omega \setminus (\mathrm{supp}(f^-))^\varepsilon)$.
\end{proposition}

Here, $A^\varepsilon$ denotes the $\varepsilon$-neighbourhood of a compact set $A$.

\begin{proof}
We will use a similar strategy as in the proof of \cite[Proposition 3.1]{DS}. First, suppose that $f^- \in \mathcal{M}(\partial\Omega)$ is finitely atomic. Denote by $D^- = \{ x_i: i = 1,...m \}$ the set of atoms of $f^-$, and denote by $T$ the (unique) optimal transport map from $f^+$ to $f^-$ (its existence follows from \cite[Proposition 2.6]{DS}). By Theorem \ref{thm:dweiksantambrogio}, the transport density $\sigma$ is absolutely continuous. For all $i = 1,...,m$, consider the set $T^{-1}(x_i)$, and without loss of generality suppose that each $T^{-1}(x_i)$ is represented by a single Lipschitz chart (if it were not, we can partition this set into finitely many parts with this property). Denote by $\Omega_i$ the set of all transport rays from $T^{-1}(x_i)$ to $x_i$ and notice that the sets $\Omega_i$ are disjoint (up to a set of zero Lebesgue measure). Now, for $\tau \in (0,1)$ denote
\begin{equation*}
\Omega_i^{(\tau)} = \{ (1-t) x + tx_i, \quad x \in T^{-1}(x_i), \quad t \leq \tau \}.
\end{equation*}
Note that $\Omega_i^{(\tau)} \subset \Omega_i$. Moreover, given $\varepsilon > 0$ there exists a constant $\tau_0 < 1$ depending on $\varepsilon$ such that $\Omega_i \setminus (D^-)^\varepsilon \subset \Omega_i^{(\tau_0)}$ for all $i = 1,...,m$ (for instance, we can take $\tau_0 = 1 - (2\mathrm{diam}(\Omega))^{-1} \varepsilon$). Fix $\tau \leq \tau_0$. Up to choosing a suitable coordinate system, $T^{-1}(x_i)$ is contained is a graph of a Lipschitz function $\alpha_i$. By definition, for every $y \in \Omega_i^{(\tau)}$ there exists a point $x = (s,\alpha_i(s))$ such that $y = (1-t)x + tx_i$ with $t \leq \tau$. Set $\sigma^{(\tau)} = \sigma|_{\Omega^{(\tau)}}$; then, exactly as in the proof of \cite[Proposition 3.1]{DS}, we get that
\begin{equation}
\sigma^{(\tau)}(y) = \frac{|x - x_i| f^+(x)}{(1-t)^{N-1} (x_i - x) \cdot \mathbf{n}(x)} \qquad \mbox{for all } y \in \Omega_i^{(\tau)},
\end{equation}
where $\mathbf{n}(x)$ is the inner normal unit vector at $x$. By uniform convexity of $\Omega$, for all $y \in \Omega_i^{(\tau)}$
\begin{equation}
\sigma^{(\tau)}(y) = \frac{|x - x_i| f^+(x)}{(1-t)^{N-1} (x_i - x) \cdot \mathbf{n}(x)} \leq \frac{|x - x_i| f^+(x)}{C (1-t)^{N-1} |x - x_i|^2} \leq \frac{f^+(x)}{C (1 - \tau)^{N-1} |x - x_i|}.
\end{equation}
Choose a representative of $f^+$ such that it is bounded by $\| f^+ \|_{L^\infty(\partial\Omega)}$ for all $x \in \partial\Omega$. Hence, for all $y \in \Omega_i^{(\tau)} \setminus (D^-)^\varepsilon$ we have
\begin{equation}
\sigma^{(\tau)}(y) \leq \frac{f^+(x)}{C \varepsilon (1 - \tau)^{N-1}} \leq \frac{f^+(x)}{C \varepsilon (1 - \tau_0)^{N-1}} \leq \frac{\| f^+ \|_{L^\infty(\partial\Omega)}}{C \varepsilon (1 - \tau_0)^{N-1}}.
\end{equation}
This bound does not depend on $i$ or $\tau$ (as long as $\tau \leq \tau_0$), hence it is valid for all points $y \in \Omega \setminus (D^-)^\varepsilon$ and we have
\begin{equation}\label{eq:boundonsigmatau}
\| \sigma \|_{L^\infty(\Omega \setminus (D^-)^\varepsilon)} \leq \frac{\| f^+ \|_{L^\infty(\partial\Omega)}}{C \varepsilon (1 - \tau_0)^{N-1}},
\end{equation}
where $C$ is a constant depending only on the curvature of $\partial\Omega$ and $\tau_0$ depends on $\varepsilon$.

Now, we prove the result for any target measure $f^- \in \mathcal{M}(\partial\Omega)$. Again, the optimal transport map from $f^+$ to $f^-$ exists and is unique by virtue of \cite[Proposition 2.6]{DS}, so in particular the transport density $\sigma$ is uniquely defined. Take a sequence of finitely atomic measures $f_n^-$ weakly converging to $f^-$ such that $\mathrm{supp}(f_n^-) \subset \mathrm{supp}(f^-)$. For $\tau \in (0,1)$, denote
\begin{equation*}
\Omega^{(\tau)} = \{ (1-t) x + tz, \quad z \in \mathrm{supp}(f^-), \quad x \in \partial\Omega \setminus (\mathrm{supp}(f^-))^\varepsilon, \quad t \leq \tau \}.
\end{equation*}
Given $\varepsilon > 0$, there exists a constant $\tau_0 = < 1$ depending on $\varepsilon$ such that $\Omega \setminus (\mathrm{supp}(f^-))^\varepsilon \subset \Omega^{(\tau_0)}$ (we may again take $\tau_0 = 1 - (2\mathrm{diam}(\Omega))^{-1} \varepsilon$). In particular, we can use this $\tau_0$ in the computation above for any $f_n$. Denote by $\sigma_n$ the sequence of transport densities corresponding to the (unique) optimal transport plans $\gamma_n$ between $f^+$ and $f_n^-$. By \cite[Proposition 2.4]{DS}, up to a subsequence $\gamma_n \rightharpoonup \gamma$, where $\gamma$ is an optimal transport plan between $f^+$ and $f^-$. Hence, by lower semicontinuity of the $L^\infty$ norm and equation \eqref{eq:boundonsigmatau}, we have
\begin{equation}
\| \sigma \|_{L^\infty(\Omega \setminus (\mathrm{supp}(f^-))^\varepsilon)} \leq \liminf_{n \rightarrow \infty} \| \sigma_n \|_{L^\infty(\Omega \setminus (\mathrm{supp}(f^-))^\varepsilon)} \leq \frac{\| f^+ \|_{L^\infty(\partial\Omega)}}{C \varepsilon (1 - \tau_0)^{N-1}},
\end{equation}
because the bound on $\sigma_n$ is uniform and is preserved in the limit.
\end{proof}

If the supports of $f^+$ and $f^-$ are disjoint, the bounds on $\sigma$ are actually up to the boundary. In general, this needs not be the case, as the counterexample in \cite[Section 4]{DS} shows.

\begin{corollary}\label{cor:uptotheboundary}
Suppose that $\Omega \subset \mathbb{R}^d$ is uniformly convex, where $d \geq 2$. Let $f^\pm \in L^\infty(\partial\Omega)$. Suppose that $\mathrm{supp}(f^+) \cap \mathrm{supp}(f^-) = \emptyset$. Then, if $\sigma$ is the (uniquely defined) transport density between $f^+$ and $f^-$, for all $\varepsilon > 0$ we have $\sigma \in L^\infty(\Omega)$.
\end{corollary}

\begin{proof}
We use Proposition \ref{prop:locallyfinite} twice: first, for the pair $f^+$ and $f^-$, and we get that for all $\varepsilon > 0$ we have $\sigma \in L^\infty(\Omega \setminus (\mathrm{supp}(f^-))^\varepsilon)$. Then, we use it for the pair $f^-$ and $f^+$, and we get that for all $\varepsilon > 0$ we have $\sigma \in L^\infty(\Omega \setminus (\mathrm{supp}(f^+))^\varepsilon)$. Since the supports of $f^+$ and $f^-$ are disjoint, for sufficiently small $\varepsilon > 0$ we have that
\begin{equation*}
\Omega = (\Omega \setminus (\mathrm{supp}(f^+))^\varepsilon)) \cup (\Omega \setminus (\mathrm{supp}(f^-))^\varepsilon)), 
\end{equation*}
so actually $\sigma \in L^\infty(\Omega)$.
\end{proof}

In the setting of the least gradient problem, the results above easily translate to results on local boundedness of the gradient of the solution; below, we state an analogue of Proposition \ref{prop:locallyfinite}.

\begin{corollary}
Suppose that $\Omega \subset \mathbb{R}^2$ is uniformly convex. Suppose that $g \in \mathrm{Lip}(\partial\Omega)$. Then, if $u \in BV(\Omega)$ is the (unique) solution to problem \ref{eq:leastgradientproblem} with boundary data $g$, then $u \in \mathrm{Lip}_{loc}(\Omega)$.
\end{corollary}

\begin{proof}
Since $g \in \mathrm{Lip}(\partial\Omega)$, its tangential derivative $f = \partial_\tau g$ is such that $f^\pm \in L^\infty(\partial\Omega)$. Since $f$ is atomless, by \cite[Proposition 2.5]{DS} the optimal transport plan is unique and induced by a map, so the transport density is unique; denote it by $\sigma$. We apply Proposition \ref{prop:locallyfinite} to get that $\sigma \in L^\infty(\Omega \setminus (\mathrm{supp}(f^-))^\varepsilon$ for all $\varepsilon > 0$; because $f^-$ is supported on $\partial\Omega$, this means that $\sigma$ is locally bounded in $\Omega$. Using Theorems \ref{thm:lgpbeckmannequivalence} and \ref{thm:beckmannkantorovichequivalence}, in particular the correspondence $|Du| = |p| = \sigma$, we get that $\nabla u \in L^\infty_{loc}(\Omega)$.
\end{proof}

Now, we proceed to give the main result for regularity of solutions to the least gradient problem in the case when the boundary datum has only a finite number of discontinuities. It is optimal in the sense that in general we cannot expect Sobolev regularity of the solution on any larger set.

\begin{proposition}\label{prop:locallyfinitejumpset}
Suppose that $\Omega \subset \mathbb{R}^2$ is uniformly convex. Let $g \in BV(\partial\Omega)$ and suppose that the set $D$ of its discontinuity points is finite. Suppose further that $g \in \mathrm{Lip}(\partial\Omega \setminus D)$. Set $J = \bigcup_{p,q \in D} [p,q]$; then, if $u \in BV(\Omega)$ is a solution to problem \eqref{eq:leastgradientproblem} with boundary data $g$, then $u \in \mathrm{Lip}_{loc}(\Omega \setminus J)$.
\end{proposition}

Again, by writing $u \in \mathrm{Lip}_{loc}(\Omega \setminus J)$, we mean that $\Omega \setminus J$ is an open set with finitely many connected components $(\Omega_i)_{i=1}^m$ and $u \in\mathrm{Lip}_{loc}(\Omega_i)$ for all $i = 1,...,m$. We use the same notation on $\partial\Omega$. In particular, the result $u \in \mathrm{Lip}_{loc}(\Omega \setminus J)$ means that the solution to the least gradient problem is Lipschitz in a neigbourhood of a generic point.

\begin{proof}
The assumption $g \in \mathrm{Lip}(\partial\Omega \setminus D)$ corresponds to the fact that $f = \partial_\tau g \in L^\infty(\partial\Omega \setminus D)$; in particular, $f$ is a sum of a finite sum of Dirac deltas at points in $D$ and an $L^\infty$ function. We proceed as in the proof of Theorem \ref{thm:sbvregularity} and keep the same notation. By Step 4, we know that the transport densities $\sigma_{\gamma_n}$ for $n = 2,3,4$ are absolutely continuous; we will improve on this result. 

Note that since $D$ is finite, both sets $D^+$ and $D^-$ are finite. By Step 4 of the proof of Theorem \ref{thm:sbvregularity}, $(\Pi_y)_{\#} \gamma_2$ is absolutely continuous and $(\Pi_y)_{\#} \gamma_2 \leq f^-$. Hence, $(\Pi_y)_{\#} \gamma_2 \in L^\infty(\partial\Omega)$. By Proposition \ref{prop:locallyfinite} we get that $\sigma_{\gamma_2} \in L^\infty(\Omega \backslash (D^+)^\varepsilon)$. Similarly, $(\Pi_y)_{\#} \gamma_3$ is absolutely continuous and $(\Pi_x)_{\#} \gamma_3 \leq f^+$; hence, $(\Pi_x)_{\#} \gamma_3 \in L^\infty(\partial\Omega)$, and by Proposition \ref{prop:locallyfinite} we get that $\sigma_{\gamma_3} \in L^\infty(\Omega \backslash (D^-)^\varepsilon)$.

Now, we estimate $\sigma_{\gamma_4}$. We will separate this transport density into a few parts. For $i = 1,...,m$, denote by $x_i$ the points in $D$, and by $\chi_i$ the open arcs between the points in $x_i$ and $x_{i+1}$ (with the convention that $x_{m+1} = x_1$ and $\chi_{m+1} = \chi_1$). We set
\begin{equation*}
\gamma_4 = \sum_{i=1}^m \gamma_{4,i} + \gamma_{4,0},
\end{equation*}
where $\gamma_{4,i}$ is the part of the transport taking place from $\chi_i$ to $\chi_i$, namely
\begin{equation}
\gamma_{4,i} := \gamma_4|_{\chi_i \times \chi_i} = \gamma|_{\chi_i \times \chi_i}
\end{equation}
and $\gamma_{4,0}$ is the part of $\gamma_4$ which corresponds to transport between some $\chi_i$ and $\chi_j$ for $i \neq j$. 

Since transport rays cannot intersect at an interior point, and every point in $D$ lies in at least one transport ray, there exists a neighbourhood of every point $x_i$ in $D$ in which there is no transport from $\chi_{i-1}$ to $\chi_i$; since the number of points in $D$ is finite, there exists $\delta > 0$ such that for $x \in \chi_i$ and $y \in \chi_j$ with $(x,y) \in \mathrm{supp}(\gamma)$ we have that $|x - y| \geq \delta$. Hence, by Corollary \ref{cor:uptotheboundary} $\sigma_{\gamma_{4,0}} \in L^\infty(\Omega)$.

Finally, we need to estimate $\sigma_{\gamma_{4,i}}$ away from the boundary of $\Omega$. Fix $\varepsilon > 0$. By Step 4 of the proof of Theorem \ref{thm:sbvregularity}, $(\Pi_x)_{\#} \gamma_4$ is absolutely continuous and $(\Pi_x)_{\#} \gamma_4 \leq f^+$, so $(\Pi_x)_{\#} \gamma_4 \in L^\infty(\partial\Omega)$; similarly, we get that $(\Pi_x)_{\#} \gamma_{4,i} \in L^\infty(\partial\Omega)$. By Proposition \ref{prop:locallyfinite} we get that $\sigma_{\gamma_4,i} \in L^\infty(\Omega \backslash (\chi_i)^\varepsilon)$.

Collecting the estimates on the transport density of $\gamma_2, \gamma_3, \gamma_{4,0}$ and $\gamma_{4,i}$, we get that 
\begin{equation}
\sigma_{\overline{\gamma}} - \sigma_{\gamma_1} = \sigma_{\gamma_2} + \sigma_{\gamma_3} + \sigma_{\gamma_{4,0}} + \sum_{i=1}^m \sigma_{\gamma_{4,i}} \in L^\infty(\Omega \setminus (\partial\Omega)^\varepsilon).
\end{equation}
Since $\sigma_{\gamma_1}$ is concentrated on $J$, we get that $\sigma_{\overline{\gamma}} \in L^\infty(\Omega \setminus (J \cup (\partial\Omega)^\varepsilon))$. Using Theorems \ref{thm:lgpbeckmannequivalence} and \ref{thm:beckmannkantorovichequivalence}, in particular the correspondence $|Du| = |p_{\overline{\gamma}}| = \sigma_{\overline{\gamma}}$, we get that $\nabla u \in L^\infty_{loc}(\Omega \backslash J)$.
\end{proof}

Obviously, in some special cases a variation of the proof above may be used to prove regularity of the solution up to the boundary (except for the discontinuity set $D$). This can be done when some of the expressions have a simple form: for instance, in the Brothers example (Example \ref{ex:brothers}), for any solution $u$ to problem \eqref{eq:leastgradientproblem} we have that $\gamma_2$ and $\gamma_3$ disappear and marginals of $\gamma_{4,i}$ are smooth, so we may use \cite[Proposition 3.5]{DS} to get that $\sigma_{\gamma_{4,i}} \in L^\infty(\Omega)$; hence, $u \in \mathrm{Lip}(\Omega \setminus J)$. On the other hand, we cannot expect better regularity of $u$ than the one given in Proposition \ref{prop:locallyfinitejumpset}. Lipschitz continuity of $u$ may break down in several ways: when $\gamma_2$ (or $\gamma_3$) does not disappear, then $u$ cannot be Lipschitz in the neighbourhood of the discontinuity point in its support; for instance, consider boundary data $g$ which are increasing on $\partial\Omega \setminus \{ p \}$ with a drop in value at $p$. Then, it is clear that the solution is not Lipschitz around $p$. Also, the counterexample in \cite[Section 4]{DS} shows that $\sigma_{\gamma_{4,i}}$ may fail to be bounded near $\partial\Omega$.

Finally, let us comment on the anisotropic case. All the results in this Section are also valid for any strictly convex norm $\varphi$. We used strict convexity of $\varphi$ on several occasions: apart from the equivalence described in Section \ref{sec:preliminaries}, strict convexity of $\varphi$ is required in order for Theorem \ref{thm:dweiksantambrogio} to be valid and the fact that transport rays are line segments is used in Step 4 of the proof of Theorem \ref{thm:sbvregularity} to prove that the transport density $\sigma_{\gamma_1}$ is concentrated on a set of Hausdorff dimension one.

\section{Stability}\label{sec:stability}

Our aim in this Section is to prove a general stability result for solutions of the least gradient problem using optimal transport techniques. This issue has been first studies by Miranda in \cite{Mir} using the concept of least gradient functions, i.e. functions which are solutions to the least gradient problem for the boundary data equal to their trace (actually, the author uses a slightly different definition which is also valid on unbounded domains). Miranda proved that an $L^1$ limit of least gradient functions is itself a least gradient function. Since then, Miranda's theorem was often used to prove existence of solutions to \eqref{eq:leastgradientproblem} in the following way (see for instance \cite{GRS2017NA,Gor2018CVPDE,Gor2019IUMJ,RS}): approximate the boundary data $g \in L^1(\partial\Omega)$ by a well-chosen sequence of functions $g_n \in L^1(\partial\Omega)$ and take the solutions $u_n \in BV(\Omega)$ to \eqref{eq:leastgradientproblem} for the approximate boundary data $g_n$. Then, prove that we can pass to the limit $u_n \rightarrow u$ in $L^1(\Omega)$, so that $u \in BV(\Omega)$ is a least gradient function; then, prove that for this special choice of the approximating sequence the trace of the limit equals $g$. This implies that $u$ is a solution to the least gradient problem with boundary data $g$. A  similar technique, involving also an approximation of the domain, was used in \cite{GRS2017NA,RS} to prove existence of solutions on convex polygons under some admissibility conditions on the boundary data.

However, since the trace operator is not continuous with respect to $L^1$ convergence, this reasoning depends on choosing an approximating sequence which is best suited to the problem at hand. In general, Miranda's theorem does not imply that solutions to \eqref{eq:leastgradientproblem} for boundary data $g_n$ converges to a solution to \eqref{eq:leastgradientproblem} for boundary data $g$ whenever $g_n \rightarrow g$ in $L^1(\partial\Omega)$; indeed, it was shown in \cite{ST} that there exist boundary data in $L^\infty(\partial\Omega)$ for which there is no solution. Therefore, to the best of the author's knowledge, the results in this Section are the first stability results for least gradient problem which do not require a special form of the approximating sequence. Note that since we use optimal transport techniques, the boundary data necessarily lie in $BV(\partial\Omega)$, so the counterexample from \cite{ST} does not apply; on the other hand, this means that our results are close to optimal.

The first result in this Section is an estimate on the total variation of a solution to the least gradient problem in terms of the total variation of its boundary datum. To the best of the author's knowledge, this type of estimate appeared in the literature only once, in \cite[Lemma 2.13]{RS2}. However, the authors of \cite{RS} prove it under very restrictive conditions on boundary data. Here, we give a much simpler proof of this result using optimal transport methods; moreover, we do not require any structural assumptions on $\Omega$ (apart from the usual assumption of convexity) and $g$, and the constant we obtain is sharp. 

\begin{proposition}\label{prop:totalvariationestimate}
Suppose that $u \in BV(\Omega)$ is a solution to the least gradient problem \eqref{eq:leastgradientproblem} for boundary data $g \in BV(\partial\Omega)$. Then,
\begin{equation}\label{eq:estimateontotalvariation}
|Du|(\Omega) \leq \frac{\mathrm{diam}(\Omega)}{2} |Dg|(\partial\Omega).
\end{equation}
\end{proposition}

\begin{proof}
As usual, denote by $p$ the corresponding solution to the Beckmann problem \eqref{eq:beckmannproblem}, by $\gamma$ the solution the corresponding Monge-Kantorovich problem \eqref{eq:kantorovich}, and by $\sigma_\gamma$ its transport density. By formula \eqref{eq:definitionofsigmaversiontwo}, we have
\begin{equation*}
|Du|(\Omega) = |p|(\overline{\Omega}) = \sigma_\gamma(\overline{\Omega}) = \int_{\overline{\Omega} \times \overline{\Omega}} \mathcal{H}^1([x,y] \cap \overline{\Omega}) \, \mathrm{d}\gamma(x,y) \leq 
\end{equation*}
\begin{equation*}
\leq \mathrm{diam}(\Omega) \gamma(\overline{\Omega} \times \overline{\Omega}) = \mathrm{diam}(\Omega) f^+(\partial\Omega) = \frac{\mathrm{diam}(\Omega)}{2} |f|(\partial\Omega) = \frac{\mathrm{diam}(\Omega)}{2} |Dg|(\partial\Omega).
\end{equation*}
\end{proof}

The following Example shows that the the constant in inequality \eqref{eq:estimateontotalvariation} is sharp.

\begin{example}
Let $\Omega = B(0,1)$. Take $g \in BV(\partial\Omega)$ given by the formula $g(x,y) = \chi_{\{ (x,y) \in \partial\Omega: \, y > 0 \}}$. Then, $u \in BV(\Omega)$, the solution to problem \eqref{eq:leastgradientproblem} with boundary data $g$, is given by the formula $u(x,y) = \chi_{\{ (x,y) \in \Omega: \, y > 0 \}}$. In particular, we have
\begin{equation*}
|Du|(\Omega) = 2 = \frac{\mathrm{diam}(\Omega)}{2} |Dg|(\partial\Omega),
\end{equation*}
so we have equality in \eqref{eq:estimateontotalvariation}.
\end{example}

We proceed to prove some stability results for least gradient functions. Our main tool will be a stability result for optimal transport plans, see \cite[Theorem 1.50]{San2015}. For simplicity, we state it here for the Euclidean cost.

\begin{theorem}\label{thm:santambrogiostability}
Suppose that $X$ and $Y$ are compact metric spaces. Suppose that $\gamma_n \in \mathcal{P}(X \times Y)$ is a sequence of optimal transport plans between $\mu_n$ and $\nu_n$. If $\gamma_n \rightharpoonup \gamma$, then $\mu_n \rightharpoonup \mu$, $\nu_n \rightharpoonup \nu$ and $\gamma$ is an optimal transport plan between $\mu$ and $\nu$.
\end{theorem}

Here, $\mathcal{P}(X \times Y)$ denotes the set of all probability measures on $X \times Y$. In particular, Theorem \ref{thm:santambrogiostability} implies that the infimum in \eqref{eq:kantorovich} for $f^+ = \mu$ and $f^- = \nu$ is the limit of the infima for $f^+ = \mu_n$ and $f^- = \nu_n$, see \cite[Theorem 1.51]{San2015}.

We will use Theorem \ref{thm:santambrogiostability} to obtain several stability results in the least gradient problem. In the first result, we will keep the domain fixed, i.e. take $X = Y = \overline{\Omega}$, and prove that on strictly convex domains the convergence of optimal transport plans given by Theorem \ref{thm:santambrogiostability} corresponds to strict convergence of solutions to the least gradient problem.

\begin{theorem}\label{thm:stability}
Suppose that $\Omega \subset \mathbb{R}^2$ is strictly convex. Suppose that $g,g_n \in BV(\partial\Omega)$ and that $g_n \rightarrow g$ strictly in $BV(\partial\Omega)$. Suppose that $u_n \in BV(\Omega)$ are solutions to problem \eqref{eq:leastgradientproblem} with boundary data $g_n$. Then, there exists $u \in BV(\Omega)$, a solution to problem \eqref{eq:leastgradientproblem}, such that (possibly after passing to a subsequence) we have $u_n \rightarrow u$ strictly in $BV(\Omega)$.
\end{theorem}

Note that because $\Omega$ is strictly convex, we have existence of solutions $u_n \in BV(\Omega)$ to problem \eqref{eq:leastgradientproblem} with boundary data $g_n \in BV(\partial\Omega)$, see \cite{DS,Gor2018CVPDE}.  

\begin{proof}
{\bf Step 1.} We will modify the sequence $g_n$ in order to be able to apply Theorem \ref{thm:santambrogiostability}. First, suppose that $|Dg|(\partial\Omega) > 0$; since $|Dg_n|(\partial\Omega) \rightarrow |Dg|(\partial\Omega)$, for sufficiently large $n$ we also have $|Dg_n|(\partial\Omega) > 0$. We set
\begin{equation*}
\overline{g}_n = \frac{|Dg|(\partial\Omega)}{|Dg_n|(\partial\Omega)} g_n.
\end{equation*}
Notice that
\begin{equation*}
\| \overline{g}_n - g \|_{L^1(\partial\Omega)} \leq \| g_n - g \|_{L^1(\partial\Omega)} + \| \overline{g}_n - g_n \|_{L^1(\partial\Omega)} \leq \| g_n - g \|_{L^1(\partial\Omega)} + \| g_n \|_{L^1(\partial\Omega)} \bigg| \frac{|Dg|(\partial\Omega)}{|Dg_n|(\partial\Omega)} - 1 \bigg|, 
\end{equation*}
hence $\overline{g}_n \rightarrow g$ in $L^1(\partial\Omega)$. Moreover, by definition we have $|D\overline{g}_n|(\partial\Omega) = |Dg|(\partial\Omega)$. 

Since we assumed that $u_n$ are solutions to the least gradient problem with boundary data $g_n$, then the functions
\begin{equation*}
\overline{u}_n = \frac{|Dg|(\partial\Omega)}{|Dg_n|(\partial\Omega)} u_n
\end{equation*}
are solutions to problem \eqref{eq:leastgradientproblem} with boundary data $\overline{g}_n$. The sequence $\overline{u}_n$ is uniformly bounded in $BV(\Omega)$: since $g_n \rightarrow g$ strictly in $BV(\partial\Omega)$, we have that $\sup_n \| g_n \|_\infty < \infty$. Hence, by the maximum principle (see for instance \cite[Theorem 5.1]{HKLS}), we have that $\sup_n \| u_n \|_\infty < \infty$. The total variations are uniformly bounded by Proposition \ref{prop:totalvariationestimate}. Hence, possibly passing to a subsequence, we have that $\overline{u}_n \rightarrow u$ weakly* in $BV(\Omega)$. In the next steps, we will upgrade the weak* convergence to strict convergence; in particular, this will imply that the trace of $u$ equals $g$.

{\bf Step 2.} Denote $\overline{f}_n = \partial_\tau \overline{g}_n$ and $f = \partial_\tau g$. Notice that since $\overline{g}_n \rightarrow g$ strictly in $BV(\partial\Omega)$, we also have $\overline{f}_n^+ \rightharpoonup f^+$ and $\overline{f}_n^- \rightharpoonup f^-$. Let $p_n = R_{-\frac{\pi}{2}} Du_n \in \mathcal{M}(\overline{\Omega}; \mathbb{R}^2)$ be a sequence of solutions to the Beckmann problem with boundary data $\overline{f}_n$. Possibly passing to a subsequence, we have $p_n \rightarrow p$ weakly in $\mathcal{M}(\overline{\Omega}; \mathbb{R}^2)$. Moreover, by uniqueness of the weak limit, we have $p = R_{-\frac{\pi}{2}} Du$ in $\Omega$; however, we do not yet know if $p$ gives no mass to the boundary.

{\bf Step 3.} Denote by $\gamma_n$ the optimal transport plan between $\overline{f}_n^+$ and $\overline{f}_n^-$ induced by the solution $p_n$ to the Beckmann problem. Since $\overline{\Omega} \times \overline{\Omega}$ is compact and all $\gamma_n$ have the same measure, by Prokhorov's theorem we have that $\gamma_n \rightharpoonup \gamma$ (possibly after passing to a subsequence). Hence, we are in position to apply Theorem \ref{thm:stability}; in the original formulation, $\gamma_n$ are probability measures, but we may apply it since all of them have the same mass (because all the measures $f_n^+$ have the same mass). We get that $\gamma$ is an optimal transport plan between $f^+$ and $f^-$. Since $\Omega$ is strictly convex, we get that
\begin{equation*}
\sigma_\gamma(\partial\Omega) = \int_{\overline{\Omega} \times \overline{\Omega}} \mathcal{H}^1(\partial\Omega \cap [x,y]) \, \mathrm{d}\gamma(x,y) = 0.
\end{equation*}
Hence, $\gamma$ corresponds to a minimiser of the Beckmann problem $p_\gamma$ which gives no mass to the boundary. Moreover, by equation \eqref{eq:definitionofpgamma} since $\gamma_n \rightharpoonup \gamma$, we have that $p_n \rightharpoonup p_\gamma$ weakly in $\mathcal{M}(\overline{\Omega}; \mathbb{R}^2)$; by the uniqueness of the weak limit, we have $p = p_\gamma$. Hence, $p$ is a minimiser of the Beckmann problem with boundary data $f$ and we have $|p|(\partial\Omega) = \sigma_\gamma(\partial\Omega) = 0$. This in turn implies that $u$ is a solution of the least gradient problem with boundary data $g$. Moreover, we have 
\begin{equation*}
\lim_{n \rightarrow \infty} |D\overline{u}_n|(\Omega) = \lim_{n \rightarrow \infty} |p_n|(\Omega) = \lim_{n \rightarrow \infty} |p_n|(\overline{\Omega}) = |p|(\overline{\Omega}) = |p|(\Omega) = |Du|(\Omega),
\end{equation*}
hence $\overline{u}_n \rightarrow u$ strictly in $BV(\Omega)$.

{\bf Step 4.} Now, we go back to the original sequence $u_n$. Notice that
\begin{equation*}
\| u_n - u \|_{L^1(\Omega)} \leq \| \overline{u}_n - u \|_{L^1(\Omega)} + \| \overline{u}_n - u_n \|_{L^1(\Omega)} \leq \| \overline{u}_n - u \|_{L^1(\Omega)} + \| u_n \|_{L^1(\Omega)} \bigg| \frac{|Dg|(\partial\Omega)}{|Dg_n|(\partial\Omega)} - 1 \bigg|, 
\end{equation*}
so $u_n \rightarrow u$ in $L^1(\Omega)$. Moreover,
\begin{equation*}
\lim_{n \rightarrow \infty} |Du_n|(\Omega) = \lim_{n \rightarrow \infty} \frac{|Dg|(\partial\Omega)}{|Dg_n|(\partial\Omega)} \cdot \lim_{n \rightarrow \infty} |D \overline{u}_n|(\Omega) = |Du|(\Omega),
\end{equation*}
so $u_n \rightarrow u$ strictly in $BV(\Omega)$, which finishes the proof.
\end{proof}

Since the formulation of Theorem \ref{thm:santambrogiostability} is quite general, we may adapt the argument used in the proof of Theorem \ref{thm:stability} to other settings. In the remainder of this Section, we will present a few results obtained in this way. First, let us focus on approximation of a convex domain $\Omega$ in the Hausdorff distance by a decreasing sequence of strictly convex domains $\Omega_n$. We will show that then the solutions to the approximate least gradient problems converge (after restriction to $\Omega$) to a solution of the original problem. This type of approximation has been used to prove existence of solutions on a convex domain, namely a rectangle, in the proof of \cite[Theorem 4.1]{GRS2017NA} (see also \cite[Theorem 3.8]{RS}). However, these results were proved in very specific settings; on the other hand, the result we give in Theorem \ref{thm:domainstability} does not depend on the form of the approximating sequence $\Omega_n$ and allows for arbitrary boundary data $g \in BV(\partial\Omega)$.

Given a convex domain $\Omega$ and a strictly convex domain $\Omega'$ such that $\Omega \subset \Omega'$, we denote by $\pi: \partial\Omega' \rightarrow \partial\Omega$ the (unique) orthogonal projection onto the closed convex set $\overline{\Omega}$. Since we assumed $\Omega \subset \Omega'$, the image of this map necessarily equals $\partial\Omega$. Moreover, by strict convexity of $\Omega'$ for $x,y \in \partial\Omega$ and any points $x' \in \pi^{-1}(x)$ and $y \in \pi^{-1}(y)$ the line segments $[x,x']$ and $[y,y']$ may intersect only if $x = y$. Now, for $g \in BV(\partial\Omega)$, we define $g': \partial\Omega' \rightarrow \mathbb{R}$ by the formula
\begin{equation*}
g'(x) = g(\pi(x)).
\end{equation*}
This definition requires a bit of clarification. Since $g \in BV(\Omega)$, $g$ admits a representative with the following properties: it is continuous everywhere except for its jump set, which is countable, and for every point in the jump set the value of $g$ equals the mean of its one-side limits. We define $g'$ using this representative. Clearly, $g' \in L^\infty(\partial\Omega')$; as we will see in the Lemma below, it actually lies in $BV(\partial\Omega')$.



\begin{lemma}\label{lem:approximationdomain}
Let the sets $\Omega,\Omega'$ and the functions $g,g'$ be defined as above. Then $g' \in BV(\partial\Omega')$. Furthermore, $|Dg'|(\partial\Omega') = |Dg|(\partial\Omega)$.
\end{lemma}

\begin{proof}
We will use the one-dimensional definition of BV functions. We say that $\{ p_0, p_1, ..., p_k \}$ is a partition of $\partial\Omega'$, if the points are ordered in such a way that on one of the arcs on $\partial\Omega'$ between $p_i$ and $p_{i+1}$ there are no other points from this set. We complement this by setting $p_{k+1} = p_0$. Moreover, we make an analogous definition on $\partial\Omega$, and denote by $\mathcal{P}$ the family of partitions of $\partial\Omega$ and by $\mathcal{P}'$ the family of partitions of $\partial\Omega'$. Then, we have
\begin{equation*}
|Dg'|(\partial\Omega') = \sup_{\mathcal{P}'} \sum_{i = 0}^{k} |g'(p_{i+1}) - g'(p_i)| = \sup_{\mathcal{P}'} \sum_{i = 0}^{k} |g(\pi(p_{i+1}) - g(\pi(p_i))| \leq |Dg|(\partial\Omega),
\end{equation*}
because if $\{ p_i \}$ is a partition of $\partial\Omega'$, then $\{ \pi(p_i) \}$ is a partition of $\partial\Omega$ (with possibly some points being equal); this is immediate if we recall that for any $x,y \in \partial\Omega$ the line segments between $x,y$ and points in their preimages cannot intersect unless $x = y$.

On the other hand, take any $\varepsilon > 0$ and fix a partition $\{ q_i \}$ of $\partial\Omega$ which is almost optimal, i.e.
\begin{equation*}
|Dg|(\partial\Omega) \leq \sum_{i = 0}^{k} |g(q_{i+1}) - g(q_i)| + \varepsilon.
\end{equation*}
Then, notice that the value of $g'$ is the same for all points in the preimage of any given point and fix any points $p_i \in \pi^{-1}(q_i)$. Hence,
\begin{equation*}
|Dg|(\partial\Omega) \leq \sum_{i = 0}^{k} |g(q_{i+1}) - g(q_i)| + \varepsilon = \sum_{i = 0}^{k} |g'(p_{i+1}) - g'(p_i)| + \varepsilon \leq |Dg'|(\partial\Omega') + \varepsilon,
\end{equation*}
because using the same argument as before we see that if $\{ q_i \}$ is a partition of $\partial\Omega$, then $\{ p_i \}$ is a partition of $\partial\Omega'$. Since $\varepsilon > 0$ was arbitrary, we get that $g' \in BV(\partial\Omega')$ and that $|Dg'|(\partial\Omega') = |Dg|(\partial\Omega)$.
\end{proof}

We will apply the above results to a sequence of approximations of the original domain $\Omega$. Namely, suppose that $\Omega_n$ is a decreasing sequence of open, bounded, strictly convex sets. Suppose additionally that $\mathrm{dist}_H(\partial\Omega_n, \partial\Omega) \rightarrow 0$, i.e. the Hausdorff distance between $\Omega_n$ and $\Omega$ converges to zero. Then, we set $\pi_n: \partial\Omega_n \rightarrow \partial\Omega$ to be the projection onto the closed convex set $\overline{\Omega}$ and set $g_n(x) = g(\pi_n(x))$ for any $x \in \partial\Omega_n$. By Lemma \ref{lem:approximationdomain}, whenever $g \in BV(\partial\Omega)$, we have $g_n \in BV(\partial\Omega_n)$. Hence, the tangential derivatives $f_n = \partial_\tau g_n$ are finite measures; in the Lemma below, we prove that they converge weakly to the tangential derivative $f= \partial_\tau g$.

\begin{lemma}\label{lem:approximatingfn}
With $f_n$ as defined above, up to a subsequence we have $f_n^\pm \rightharpoonup f^\pm$ weakly in $\mathcal{M}(\overline{\Omega_1})$.
\end{lemma}

\begin{proof}
First, notice that by construction of $g_n$ we have $f = (\pi_n)_\# f_n$. It is sufficient to show that for any open arc $\Gamma \subset \partial\Omega$ we have $f(\Gamma) = f_n(\pi_n^{-1}(\Gamma))$. Assume that the endpoints of $\Gamma$ are $p$ and $q$; then, up to choosing an orientation of $\partial\Omega$, we have $f(\Gamma) = g(p) - g(q)$ (to be exact, in this formula and the next we have one-sided limits of $g$ at $p$ and $q$). By properties of $\pi_n$, we also have that $\pi_n^{-1}(\Gamma)$ is an open arc on $\partial\Omega_n$ with endpoints $p_n$ and $q_n$; as before, we have $f_n(\pi_n^{-1}(\Gamma)) = g_n(p_n) - g_n(q_n)$. But this implies
\begin{equation*}
f_n(\pi_n^{-1}(\Gamma)) = g_n(p_n) - g_n(q_n) = g(\pi(p_n)) - g(\pi(q_n)) = g(p) - g(q) = f(\Gamma),
\end{equation*}
hence $f = (\pi_n)_\# f_n$.

Now, suppose that $\varphi \in \mathrm{Lip}(\overline{\Omega_1})$ with Lipschitz constant $L$. Then,
\begin{equation*}
\bigg| \int_{\overline{\Omega_1}} \varphi \, \mathrm{d}f_n - \int_{\overline{\Omega_1}} \varphi \, \mathrm{d}f \bigg| = \bigg| \int_{\overline{\Omega_1}} \varphi \, \mathrm{d}f_n - \int_{\overline{\Omega_1}} \varphi \, \mathrm{d} ((\pi_n)_\# f_n) \bigg| = \bigg| \int_{\overline{\Omega_1}} \varphi \, \mathrm{d}f_n - \int_{\overline{\Omega_1}} \varphi \circ \pi_n \, \mathrm{d}f_n \bigg| =
\end{equation*}
\begin{equation*}
= \bigg| \int_{\overline{\Omega_1}} ( \varphi - \varphi \circ \pi_n) \, \mathrm{d}f_n \bigg| \leq \int_{\overline{\Omega_1}} |\varphi - \varphi \circ \pi_n| \, \mathrm{d} |f_n| \leq \int_{\overline{\Omega_1}} L \, \mathrm{dist}_H(\partial\Omega_n, \partial\Omega) \, \mathrm{d} |f_n|,
\end{equation*}
which goes to zero as $n \rightarrow \infty$, because $\mathrm{dist}_H(\partial\Omega_n, \partial\Omega) \rightarrow 0$ and by Lemma \ref{lem:approximationdomain} we have $|f_n|(\overline{\Omega_1}) = |f_n|(\partial\Omega_n) = |f|(\partial\Omega)$. Hence, $f_n \rightharpoonup f$ as measures on $\overline{\Omega_1}$. Now, decompose it into $f_n = f_n^+ - f_n^-$; up to a subsequence we have $f_n^+ \rightharpoonup \mu$ and $f_n^- \rightharpoonup \nu$ weakly as measures on $\overline{\Omega_1}$, where $\mu$ and $\nu$ are positive measures with total mass equal to $\frac{1}{2} |f|(\partial\Omega)$. Since $f_n = f_n^+ - f_n^- \rightharpoonup \mu - \nu$ and $f_n \rightharpoonup f$, by uniqueness of the weak limit we have that in fact $\mu = f^+$ and $\nu = f^-$.
\end{proof}

\begin{theorem}\label{thm:domainstability}
Suppose that $\Omega$ is strictly convex, $g \in BV(\partial\Omega)$ and $\Omega_n$ and $g_n$ are constructed as above. Then, if $u_n \in BV(\Omega_n)$ are solutions to problem \eqref{eq:leastgradientproblem} with boundary data $g_n$, then on a subsequence we have $u_n|_\Omega \rightarrow u$ strictly in $BV(\Omega)$. Moreover, $u$ is a solution to the least gradient problem with boundary data $g$. 
\end{theorem}

\begin{proof}
{\bf Step 1.} Denote $\overline{f}_n = \partial_\tau \overline{g}_n$ and $f = \partial_\tau g$. By Lemma \ref{lem:approximatingfn}, we have $\overline{f}_n^+ \rightharpoonup f^+$ and $\overline{f}_n^- \rightharpoonup f^-$. Let $p_n = R_{-\frac{\pi}{2}} Du_n \in \mathcal{M}(\overline{\Omega}; \mathbb{R}^2)$ be a sequence of solutions to the Beckmann problem on $\Omega_n$ with boundary data $\overline{f}_n$. As in the proof of Theorem \ref{thm:stability}, possibly passing to subsequences, we have $p_n|_{\overline{\Omega}} \rightarrow p$ weakly in $\mathcal{M}(\overline{\Omega}; \mathbb{R}^2)$ and $u_n|_\Omega \rightarrow u$ in $L^1(\Omega)$. Moreover, by uniqueness of the weak limit, we have $p = R_{-\frac{\pi}{2}} Du$ in $\Omega$; however, we do not yet know if $p$ gives no mass to the boundary.

{\bf Step 2.} Denote by $\gamma_n$ the optimal transport plan between $\overline{f}_n^+$ and $\overline{f}_n^-$ induced by the solution $p_n$ to the Beckmann problem on $\Omega_n$. Since $\overline{\Omega_1} \times \overline{\Omega_1}$ is compact and all $\gamma_n$ have the same measure, by Prokhorov's theorem we have that $\gamma_n \rightharpoonup \gamma$ (possibly after passing to a subsequence). We apply Theorem \ref{thm:stability} and get that $\gamma \in \mathcal{M}^+(\overline{\Omega_1} \times \overline{\Omega_1})$ is an optimal transport plan between $f^+$ and $f^-$. Since all the transport rays corresponding to $\gamma$ are inside the convex hull of $\partial\Omega$, by convexity of $\Omega$ we may require that $\gamma \in \mathcal{M}^+(\overline{\Omega} \times \overline{\Omega})$, so $\gamma$ is an optimal transport plan between $f^+$ and $f^-$ in the Monge-Kantorovich problem defined on $\overline{\Omega}$. Moreover,
\begin{equation*}
\sigma_\gamma(\partial\Omega) = \int_{\overline{\Omega} \times \overline{\Omega}} \mathcal{H}^1(\partial\Omega \cap [x,y]) \, \mathrm{d}\gamma(x,y) = 0.
\end{equation*}
As in the proof of Theorem \ref{thm:stability}, this implies that $p$ is a minimiser of the Beckmann problem on $\Omega$ with boundary data $f$ and we have $|p|(\partial\Omega) = \sigma_\gamma(\partial\Omega) = 0$. This in turn implies that $u$ is a solution of the least gradient problem with boundary data $g$. Moreover, we have 
\begin{equation*}
|Du|(\Omega) \leq \liminf_{n \rightarrow \infty} |D(u_n|_\Omega)|(\Omega) = \liminf_{n \rightarrow \infty} |p_n|_\Omega|(\Omega) \leq \lim_{n \rightarrow \infty} |p_n|(\overline{\Omega_n}) = |p|(\overline{\Omega}) = |p|(\Omega) = |Du|(\Omega),
\end{equation*}
hence $u_n|_\Omega \rightarrow u$ strictly in $BV(\Omega)$.
\end{proof}

A simple application of the above result is that if we want to directly compute the minimiser or study its structure, it may sometimes be more convenient to approximate the domain by a sequence of domains with smooth boundary. Moreover, given a boundary datum $g \in BV(\partial\Omega)$ we may combine Theorems \ref{thm:stability} and \ref{thm:domainstability} to construct an approximating sequence with better regularity properties than $g \circ \pi_n$.

Finally, let us comment on the case when $\Omega$ is a domain which is convex but not strictly convex. The least gradient problem for such domains has been first studied on a rectangle in \cite{GRS2017NA}, and then on polygonal domains in \cite{RS} and \cite{RS2}. On non-strictly convex domains, existence of solutions may fail even in the simplest settings: suppose that $\Omega$ is a polygon and $g$ equals one on one of its sides (denoted by $l$) and zero on all the other sides. Then, $f = \partial_\tau g$ is the difference of two Dirac deltas at the endpoints of $l$. We easily compute the (unique) optimal transport plan $\gamma$ and notice that it gives nonzero mass to $l$. But if a solution to problem \eqref{eq:leastgradientproblem} existed, there would be a corresponding optimal transport plan which gives no mass to the boundary, contradiction. Hence, in order to prove existence of solutions to problem \eqref{eq:leastgradientproblem}, the authors introduce different sets of admissibility conditions. Let us focus on one such condition; when $g$ restricted to every maximal line segment $l_i \subset \partial\Omega$ is monotone.

\begin{corollary}\label{cor:convexexistence}
Suppose that $\Omega \subset \mathbb{R}^2$ is convex. Denote by $l_i$ the maximal line segments $l_i \subset \partial\Omega$. Suppose that $g \in BV(\partial\Omega)$ is such that for every $i$ the function $g$ is continuous at endpoints of $l_i$ and $g|_{l_i}$ is monotone. Then: \\
(1) There exists a solution to problem \eqref{eq:leastgradientproblem}; \\
(2) Let $g_n \in BV(\partial\Omega)$ be such that $g_n \rightarrow g$ strictly in $BV(\partial\Omega)$. Suppose that $u_n \in BV(\Omega)$ are solutions to problem \eqref{eq:leastgradientproblem} with boundary data $g_n$. Then, there exists $u \in BV(\Omega)$, a solution to problem \eqref{eq:leastgradientproblem}, and (possibly after passing to a subsequence) we have $u_n \rightarrow u$ strictly in $BV(\Omega)$.
\end{corollary}

\begin{proof}
(1) In light of the discussion in Section \ref{sec:preliminaries}, it is sufficient to prove that there exists an optimal transport plan between $f^\pm = (\partial_\tau g)_\pm$ which gives no mass to the boundary. Our assumption on $g$ implies that $f$ has no atoms at endpoints of $l_i$ and that for each $i$ we either have $f^+(l_i) = 0$ or $f^-(l_i) = 0$. Hence,
\begin{equation*}
\sigma_\gamma(\partial\Omega) = \int_{\overline{\Omega} \times \overline{\Omega}} \mathcal{H}^1(\partial\Omega \cap [x,y]) \, \mathrm{d}\gamma(x,y) = 0;
\end{equation*}
to see this, note that whenever $[x,y]$ is not a subset of any $l_i$, we have $\mathcal{H}^1(\partial\Omega \cap [x,y]) = 0$. On the other hand, if $[x,y] \subset l_i$ for some $i$, then 
\begin{equation*}
0 \leq \gamma(l_i \times l_i) \leq \min(\gamma(l_i \times \overline{\Omega}), \gamma(\overline{\Omega} \times l_i)) = \min(f^+(l_i), f^-(l_i)) = 0.
\end{equation*}
Hence, $\sigma_\gamma(\partial\Omega) = 0$, so we may construct a solution to \eqref{eq:leastgradientproblem}.

(2) We replicate the proof of Theorem \ref{thm:stability}. Note that strict convexity of $\Omega$ was used in the proof only once, in order to conclude that $\sigma_\gamma(\partial\Omega) = 0$. But our assumptions on $g$ guarantee this, as we just proved in point (1). 
\end{proof}

Also, notice that using the optimal transport framework enabled us to easily handle the case of arbitrary convex domains, while the analysis in \cite{RS} is restricted to polygonal domains. The result presented above suggests that some parts of the analysis performed in \cite{RS,RS2} could be not only replicated, but also generalised to arbitrary convex domains using optimal transport techniques. We will focus on one more such instance: approximation of a convex domain $\Omega$ by a decreasing sequence of strictly convex domains $\Omega_n$, with suitable approximation of the boundary data.

\begin{corollary}\label{cor:convexstability}
Suppose that $\Omega$ convex. Denote by $l_i$ the maximal line segments $l_i \subset \partial\Omega$. $g \in BV(\partial\Omega)$ is such that for every $i$ the function $g$ is continuous at endpoints of $l_i$ and $g|_{l_i}$ is monotone. Suppose that $\Omega_n$ and $g_n$ are constructed as above. Then, if $u_n \in BV(\Omega_n)$ are solutions to problem \eqref{eq:leastgradientproblem} with boundary data $g_n$, then on a subsequence we have $u_n|_\Omega \rightarrow u$ strictly in $BV(\Omega)$. Moreover, $u$ is a solution to the least gradient problem with boundary data $g$. 
\end{corollary}

\begin{proof}
We replicate the proof of Theorem \ref{thm:domainstability}. Again, strict convexity of $\Omega$ was used in the proof only once, in order to conclude that $\sigma_\gamma(\partial\Omega) = 0$. But our assumptions on $g$ guarantee this, as we just proved in point (1) of Corollary \ref{cor:convexexistence}.
\end{proof}

Finally, let us comment on the anisotropic case. All the results in this Section are also valid for any strictly convex norm $\varphi$. We used strict convexity of $\varphi$ on several occasions: apart from the equivalence described in Section \ref{sec:preliminaries}, we used the fact that transport rays are line segments to prove that $\sigma_\gamma(\partial\Omega) = 0$ and that every solution to the Beckmann problem is of the form $p = p_\gamma$ for an optimal transport plan $\gamma$ in the proofs of Proposition \ref{prop:totalvariationestimate}, Theorem \ref{thm:stability} (Step 3), Theorem \ref{thm:domainstability} (Step 2), and Corollaries \ref{cor:convexexistence} and \ref{cor:convexstability}.

{\flushleft \bf Acknowledgements.} This work was partially supported by the   DFG-FWF project FR 4083/3-1/I4354, by the OeAD-WTZ project CZ 01/2021, and by the project 2017/27/N/ST1/02418 funded by the National Science Centre, Poland.

\bibliographystyle{plain}

\end{document}